\documentclass[11pt]{article}
\usepackage{layout,yfonts, graphicx, amsmath, amsthm, amssymb, euscript, enumerate, enumitem, bbold, bbm, mathrsfs,easyReview}

\usepackage{caption}
\usepackage{subcaption}

\usepackage[numbers]{natbib} 
\numberwithin{equation}{section}
\usepackage{color}
\definecolor{webgreen}{rgb}{0,.5,0}
\definecolor{webbrown}{rgb}{0.0,0.0,0.0}
\definecolor{emphcolor}{rgb}{0.95,0.95,0.95}

\usepackage{hyperref}
\hypersetup{%
	colorlinks=true,
	linkcolor=webbrown,
	urlcolor=webbrown,
	filecolor=webbrown,
	citecolor=webbrown,
	breaklinks=true}
\ifpdf \hypersetup{pdftex,
	bookmarksopen=true,
	bookmarksnumbered=true
} \else \hypersetup{dvips} \fi

\usepackage{blindtext} 
\usepackage{scrextend}
\addtokomafont{labelinglabel}{\sffamily}

\usepackage[margin=1in]{geometry}
\theoremstyle{plain}
\newtheorem{teor}{Theorem}[section]

\newtheorem{prop}[teor]{Proposition}
\newtheorem{lema}[teor]{Lemma}
\newtheorem{coro}[teor]{Corollary}
\newtheorem{rem}[teor]{Remark}
\newtheorem{assump}[teor]{Assumption}

\newtheoremstyle{hyp}{}{}{\itshape}{}{}{}{3pt}{}
\theoremstyle{hyp}

\newcommand{\lev}{L\'{e}vy }
\newcommand{\br}{\mathbf{b}}

\DeclareMathAlphabet{\mathpzc}{OT1}{pzc}{m}{it}

\DeclareMathOperator{\hol}{C}

\DeclareMathOperator{\expo}{e}

\newcommand{\blue}{\textcolor[rgb]{0.00,0.00,1.00}}
\newcommand{\E}{\mathbbm{E}}
\newcommand{\tmt}{\mathpzc{t}}
\newcommand{\tms}{\mathpzc{s}}

\newcommand{\R}{\mathbbm{R}}

\newcommand{\uno}{\mathbbm{1}}
\newcommand{\der}{\mathrm{d}}

\newcommand{\Pro}{\mathbbm{P}}
\newcommand{\BE}{\begin{equation}}
	\newcommand{\A}{\mathcal{A}}
	\newcommand{\EE}{\end{equation}}
\newcommand {\BA}{\begin{align}}
	\newcommand{\EA}{\end{align}}
\newcommand{\eqdef}{\raisebox{0.4pt}{\ensuremath{:}}\hspace*{-1mm}=}
\newcommand{\defeq}{=\hspace*{-1mm}\raisebox{0.4pt}{\ensuremath{:}}}

\title{\Large{\sc {A mixed singular/switching  control problem with terminal cost for   modulated  diffusion processes}}\footnote{\textbf{Funding:}  {This} study has been funded by the Russian Academic Excellence Project `5-100'. \blue{Additionally, M. Kelbert was supported by the RSF Grant number 23-21-00052}}}
\author{\large{\bf Mark Kelbert}\\ 
	\large{\bf Harold A. Moreno-Franco}\footnote{Corresponding author: hmoreno@hse.ru}\\
	\small{\it Department of Statistics and Data Analysis}\\ 
	\small{\it Laboratory of Stochastic Analysis and its Applications}\\ 
	\small{\it National Research University Higher School of Economics, Moscow, Russian Federation}}
\date{}
\allowdisplaybreaks
\begin{document}

\begin{center}
	{\LARGE\bf  {On the Bailout Dividend Problem with Periodic\\ \vspace{0.15cm}Dividend Payments and  Fixed
			 Transaction Costs\footnote{The article  was prepared within the framework of the HSE University Basic Research Program.
		 }}}\\
	\vspace{0.4cm}
		{\large  {\bf Harold A. Moreno-Franco}}\footnote{Department of Statistics and Data Analysis, Laboratory of Stochastic Analysis and its Applications, National Research University Higher School of Economics, Federation of  Russia. Email: hmoreno@hse.ru.} {\large {\bf and Jos\'e-Luis P\'erez}}\footnote{Department of Probability and Statistics, Centro de Investigaci\'on en Matem\'aticas A.C., Mexico. Email: jluis.garmendia@cimat.mx.}\\
		\vspace{0.3cm}
		{\large\it In memory of Mayya A. Prokhorova  }
\end{center}

\begin{abstract}
	\noindent We study the optimal bailout dividend problem with transaction costs for an insurance company, where shareholder payouts align with the arrival times of an independent Poisson process. In this scenario, the underlying risk model  follows a spectrally negative L\'evy process. Our analysis confirms the optimality of a periodic $(b_{1},b_{2})$-barrier policy with classical reflection at zero. This strategy involves reducing the surplus to $b_1$ when it  exceeds  $b_{2}$ at the Poisson  arrival times and pushes the surplus to 0 whenever it goes below zero.\\
	
	
	
\end{abstract}

\section{Introduction}

In the context of de Finetti's dividend problem with transaction costs and capital injection,  the aim  is to identify  an optimal strategy merging dividend payments and capital injection   in order to maximize the expected net present value (NPV) involving the subtraction of fixed transaction costs and capital injection expenses  from dividend payouts.  According to the study by Junca et al. \cite{JMP2019} (also see   \cite{WWCW2022}), an effective approach for solving this problem is through a reflected $(b_{1},b_{2})$-barrier policy, paired with  classical reflection at zero, when the (uncontrolled)  surplus process  of an insurance company follows a spectrally negative L\'evy process.  This type of strategy entails  reducing the surplus to $b_{1}$ when it exceeds $b_{2}$ (where $0\leq b_{1}<b_{2}<\infty$), and pushing the surplus to $0$ whenever it goes below zero.

Within the  insurance literature, where periodic observation models have been widely studied  (e.g., \cite{ABT2011, ACT2011}), and recognizing the practical nature of dividend-payout decisions occurring at discrete times, this paper builds on the problem mentioned above by introducing a periodic dividend constraint.  This adjustment stipulates that dividend payouts to shareholders are made on the arrival times of an independent Poisson process $N^{r}$ with intensity $r>0$, while the management of capital injection follows the classical bailout restriction. This entails continuous injection of capital over time to maintain the non-negativity of the controlled process  uniformly in time. With the underlying risk model following a spectrally negative L\'evy process, our  main aim is to establish that the periodic $(b_{1},b_{2})$-barrier policy with classical reflection at zero  is optimal for the problem of maximizing the expected NPV  across all the admissible periodic-classical strategies. The periodic $(b_{1},b_{2})$-barrier policy stipulates decreasing the surplus to $b_1$ when it exceeds $b_{2}$ at the arrival times of  $N^{r}$ and injecting capital to set the surplus at $0$ whenever it becomes negative.

In previous works, assuming a spectrally negative L\'evy model, Noba et al. \cite{NobPerYamYan} studied the bailout dividend problem with periodic dividend payments, excluding transaction costs. They established the optimality of the periodic $b$-barrier policy alongside the classical reflection at zero for this scenario. The periodic $b$-barrier policy involves   paying out to the shareholders the excess between the surplus and the barrier $b$  at the arrival times of  $N^{r}$.  Moreover, Avanzi et al. \cite{ALW2021} showed the optimality of the periodic $(b_{1},b_{2})$-barrier policy for the optimal dividend problem involving periodic dividend payments and transaction costs, without capital injection. Our paper builds on and extends these earlier findings in the field.

Since under the periodic $(b_{1},b_{2})$-barrier policy with classical reflection at zero, the surplus is pushed down  to $b_{1}$ when it exceeds the level $b_{2}$  at each arrival time of the Poisson process $N^r$, and the strategy pushes the surplus to $0$ whenever it goes below zero,  the  controlled process associated with this strategy can be constructed akin to the L\'evy process with Parisian reflection above and classical reflection below introduced in \cite{PY2018}; for more details see Subsection \ref{PR}. 
The latter appears in  studies related to problems of stochastic control with periodic strategies such as in \cite{ALW2021,NobPerYamYan,NobPerYu,PY2017}, as its fluctuation identities can efficiently determine the expected NPV of barrier strategies by means of scale functions, enabling the application of the `guess and verify' method to effectively address  those optimization problems. However, since the fluctuation identities have not been studied before for a L\'evy processes under a periodic $(b_1,b_2)$-barrier policy with classical reflection at zero, we will study these identities as a preliminary step to obtain the expected NPV of periodic $(b_{1},b_{2})$-barrier policies with the classical reflection at zero. 

Our approach for confirming optimality can be outlined as follows:
\begin{enumerate}
\item[(1)] The expected NPV under a periodic $(b_{1},b_{2})$-barrier policy with classical reflection at $0$ is explicitly written  in terms  of the $q$-scale functions; see Proposition \ref{p1}. Subsequently, we employ the `guess and verify' procedure, commonly used in the literature.

\item[(2)] The candidate optimal barrier $\pi^{0,\br^{*}}$, with $\br^{*}=(b^{*}_{1},b^{*}_{2})$, is selected based on the conjecture that its expected NPV  $v^{\br^{*}}_{\alpha}$, given in  \eqref{v1}, satisfies a smoothness condition at the point $b^{*}_{2}$ and  that $v^{\br^{*}\prime}_{\alpha}$ evaluated at $b^{*}_{1}$ becomes one. We then proceed to verify the optimality of this selected barrier strategy by showing that its expected NPV satisfies the corresponding variational inequalities.   In order to check the optimality of $\pi^{0,\br^{*}}$, it is important to highlight that we compute the expected value of the Laplace transform of the first  down-crossing time for the L\'evy process under a periodic $(b_{1},b_{2})$-barrier policy with classical reflection at $0$ mentioned before; see Remark \ref{P2.0} and  Lemma \ref{L.A.3}.  This identity will help us to confirm the monotonicity of $v^{\br^{*}}_{\alpha}$. Furthermore, checking that $v^{\br^{*}\prime}_{\alpha}$ is less or equal to the first derivative of the value function of the problem with periodic-classical strategies without  transaction costs (studied by Noba et al. \cite{NobPerYamYan}), it will be proven that $v^{\br^{*}}_{\alpha}$ has linear growth; see the proof of Lemma \ref{P2}. With these properties in mind, we  then proceed to check that $v^{\br^{*}}_{\alpha}$ satisfies the variational inequalities as mentioned earlier.
\end{enumerate}

{The rest of this} document is structured as follows: in Section \ref{intro} we introduce the optimal bailout periodic dividend problem with transaction costs studied here,  and propose a construction to determine the    L\'evy processes under a periodic $(b_1,b_2)$-barrier policy with classical reflection at zero, with $0\leq b_{1}<b_{2}$; see Subsection \ref{PR}. Moving on to Section   \ref{pre},  we discuss some preliminaries about $q$-scale functions,   essential for calculating   the expected NPV $v^{\br}_{\alpha}$ of the periodic $(b_{1},b_{2})$-barrier policy with classical reflection at zero, with $\br=(b_{1},b_{2})$. This discussion includes insights into optimal bailout dividend problems without transaction costs when dividends are made periodical in time.    Noba et al. \cite{NobPerYamYan} have previously studied this problem. This background is crucial  since the expected NPV of the  periodic $(b_{1},b_{2})$-barriers with  classical reflection at zero, is closely linked to the  expected NPV of the  periodic $b_{2}$-barriers with  classical reflection at zero. Proceeding to Section \ref{NPV1}, Proposition \ref{p1} explicitly provides the expression of $v^{\br}_{\alpha}$ in terms of scale functions. 
Subsequently, a smoothness condition for $v^{\br}_{\alpha}$ at $b_{2}$ is established. 
Lemma \ref{s1} ensures the existence of points $(b_{1},b_{2})$ that satisfy this smoothness condition.  Then, in Section \ref{select1} we select the candidate for optimal  barrier $\br^{*}=(b^{*}_{1},b^{*}_{2})$ among all points that meet the smoothness condition, as the one which satisfies that  $v^{\br^{*}\prime}_{\alpha}(b^{*}_{1})=1$  if $b^{*}_{1}>0$, or $v^{\br^{*}\prime}_{\alpha}(b^{*}_{1})<1$  if $b^{*}_{1}=0$. Lemma \ref{beh1}, supported by Lemma \ref{s1}, assures the existence of a unique  $\br^{*}=(b^{*}_{1},b^{*}_{2})$ where $v^{\br^{*}}_{\alpha}$ fulfils these conditions. Finally, in Section  \ref{sec6}, we give the rigorous verification of the optimality for the problem.

\section{Formulation of the problem}\label{intro}
	
We assume  that the surplus process of an insurance company, in the absence of control, is modelled by a spectrally negative L\'evy process $X=\{X(\tmt):\tmt\geq0\}$, which is defined on a probability space $(\Omega,\mathcal{F},\Pro)$, and $\mathbbm{F}\eqdef\{\mathcal{F}_{\tmt}:\tmt \geq0\}$ is the completed and right-continuous filtration generated by $X$. Recall that a spectrally negative L\'evy process is a process that has càdlàg paths, stationary and independent increments and that only jumps downwards. We assume that $X$ has non-monotone trajectories to avoid trivial cases.
	
Throughout this paper, we will denote by $\Pro_{x}$ the law of $X$,  when $X(0)=x\in\R$, and $\E_{x}$ represents the associated expectation operator. For convenience, we write  $\Pro\equiv\Pro_{0}$  and $\E\equiv\E_{0}$, when $X(0)=0$.
	
	 Denote by  $\psi:[0,\infty)\rightarrow\R$   to  the Laplace exponent of  $X$, i.e. 
	\[
	\E\big[\expo^{\theta X(\tmt)}\big]=:\expo^{\psi(\theta)t}, \qquad t, \theta\ge 0,
	\]
	given by the \emph{L\'evy-Khintchine formula} 
	\begin{equation*}
		\psi(\theta) \eqdef\mu\theta+\frac{\sigma^2}{2}\theta^2+ \int_{(-\infty,0)}\big(\expo^{\theta z}-1-\theta z \mathbf{1}_{\{z >-1\}}\big)\Pi(\der z), \quad \theta \geq 0,
	\end{equation*}
	where $\mu\in \R$, $\sigma \ge 0$, and $\Pi$ is the \lev measure of $X$ defined on $(-\infty,0)$ satisfying
	\begin{equation*}
		 \int_{(-\infty,0)}(1\land z^2)\Pi(\der z)<\infty.
	\end{equation*}	
It is known that  $X$ has paths of bounded variation if and only if $\sigma=0$ and $  \int_{(-1,0)} |z|\Pi(\der z)<\infty$. Its Laplace exponent is given by
\begin{equation*}
	\psi(\theta) = c \theta+ \int_{(-\infty,0)}\big(\expo ^{\theta z}-1\big)\Pi(\der z), \quad \theta \geq 0,
\end{equation*}
where $c\eqdef\gamma-  \int_{(-1,0)} z\Pi(\der z)>0$ since we omitted the case of monotone paths. For more details about general L\'evy processes, the reader can consult e.g. \cite{Kyp}.	
	\subsection{Bailout optimal dividend problem with Poissonian decision times  and  a fixed transaction  cost}
	
	Assume that the insurance company pays dividends to the shareholders at the arrival times $\mathcal{T}_r \eqdef\{T(i); i\geq 1 \}$ of a Poisson process $N^r=\{ N^r(\tmt); \tmt\geq 0\}$ with intensity $r>0$ and independent of $X$. The process $N^r$ is defined on  $(\Omega, \mathcal{F}, \mathbbm{H}, \Pro)$, where $\mathbbm{H}=\{\mathcal{H}_{t}\}_{\tmt\geq 0}$ is the right-continuous complete filtration generated by $(X,N^r)$. In this case, the dividend strategy $L_r^{\pi}$ has the following form
\begin{align}
	L_r^{\pi}(\tmt)= \int_{[0,\tmt]}\nu^{\pi}(\tms)\der N^r(\tms),\qquad\text{$\tmt\geq0$,} \label{restriction_poisson}
\end{align}
for some c\`agl\`ad process $\nu^{\pi}$ adapted  to the filtration  $\mathbbm{H}$.

Additionally, the shareholders must inject capital into the insurance company to avoid the event of ruin. The total capital injected to the insurance company until time $\tmt$ is represented by the process $R_r^{\pi}=\{R_r^{\pi}(\tmt):\tmt\geq0\}$, which is non-decreasing, right-continuous, and $\mathbbm{H}$-adapted, with $R_r^{\pi}(0-) = 0$. Notice that capital injection can be made continuously in time. In addition, the process $R_r^{\pi}$ must satisfy
\begin{equation}\label{bailout_admissible}
	\E_{x}\bigg[  \int_{[0,\infty)} \expo^{-q  \tmt} \der R_r^{\pi}(\tmt) \bigg] < \infty, \quad x \geq 0,
\end{equation}
where $q >0 $ represents the rate of discounting.
The corresponding  {controlled} process associated with the strategy $\pi=(L^{\pi}_{r},R^{\pi}_{r})$ is given by $U_r^{\pi}(0-) = X(0)$ and 
\begin{align*}
	U_r^{\pi}(\tmt) \eqdef X(\tmt) - L_r^{\pi}(\tmt) + R_r^{\pi}(\tmt), \quad t \geq 0.
\end{align*}

We denote by $\A$ the set of strategies $\pi$ satisfying \eqref{restriction_poisson}--\eqref{bailout_admissible} with the additional constraint that the surplus of the insurance company has to be non-negative uniformly in time i.e.  $U_r^{\pi}(\tmt) \geq 0$ for all $t \geq 0$ a.s. We call a strategy $\pi$ \textit{admissible} if $\pi \in \A$. 

Given $\alpha>0$ and $\beta > 1$,  the constant costs per unit of dividend payments and capital  injected respectively,  our aim is to maximize the expected net present value (NPV)
\begin{equation}\label{va1}
	v^{\pi}_{\alpha}(x) \eqdef \E_{x} \bigg[  \int_{[0, \infty)} \expo^{-q  \tmt} \der \bigg[L_r^{\pi}(\tmt)-\alpha\sum_{0\leq\tms\leq\tmt} \uno_{\{\Delta L^{\pi}_{r}(\tms)\neq0\}}\bigg] - \beta  \int_{[0, \infty)} \expo^{-q  \tmt} \der R_r^{\pi}(\tmt) \bigg], \quad x \geq 0, 
\end{equation}
over all $\pi \in \A$, where $\Delta L^{\pi}_{r}(\tms)\eqdef L^{\pi}_{r}(\tms)-L^{\pi}_{r}(\tms-)$. Hence, our goal is to find the value function of the problem
\begin{equation}\label{MAP_Value}
	v_{\alpha}(x) \eqdef  \sup_{\pi \in \A} v^{\pi}_{\alpha}(x), \quad x \geq 0,
\end{equation}
and obtain an optimal strategy $\pi^* \in \A$   whose expected NPV, $v^{\pi^{*}}_{\alpha}$, coincides with $v_{\alpha}$ if such a strategy exists.

 Throughout the paper, we assume the following additional condition.
\begin{assump}\label{assump_1_aux}
	We assume that $\psi'(0+)=\E[X(1)]>-\infty$.
\end{assump}

\subsection{Periodic-classical barrier strategies}\label{PR}

To solve the problem mentioned above, we will study some properties of  the expected NPVs of the strategies $\pi^{0,\br}$ which are  conformed by a periodic  $\br$-barrier strategy, with $\br\eqdef(b_{1},b_{2})$ such that $0\leq b_{1}<b_{2}$,  and a capital injection process $R^{0,\br}_{r}$ which is given by  the classical reflection at zero.  The reason for doing this, is because our main purpose is to verify that an optimal strategy for  the  problem \eqref{MAP_Value} is of the periodic-classical $\br$-barrier type. 

So, given the barrier $\br=(b_{1},b_{2})$, with $0\leq b_{1}<b_{2}$,  the strategy $\pi^{0,\br}$ and  the (controlled) surplus  process $U_r^{0,\br}$  are constructed as follows.

Let  $R^{0,\br}_r (\tmt)\eqdef \bigg[- \inf_{0 \leq \tms \leq \tmt} X(\tms)\bigg] \vee 0$ for $\tmt\geq 0$, then we have 
\begin{align*}
	U_r^{0,\br}(\tmt) = X(\tmt) + R^{0,\br}_r(\tmt), \quad 0 \leq t < \widehat{T}_{\br} (1) ,
\end{align*}
where $\widehat{T}_{\br}(1) \eqdef \inf\{T(i):\; X(T(i)-)+R^{0,\br}_r(T(i)-) > b_{2}\}$. Here and throughout of the paper, let $\inf \emptyset = \infty$.
The process then jumps down by $X(\widehat{T}_{\br}(1))+R^{0,\br}_r(\widehat{T}_{\br}(1))-b_{1}$  
so that $U_r^{0,\br}(\widehat{T}_{\br}(1)) = b_{1}$. For $\widehat{T}_{\br}(1) \leq t < \widehat{T}_{\br}(2)  \eqdef \inf\{T(i) > \widehat{T}_{\br}(1):\; U_r^{0,\br}(T(i) -) > b_{2}\}$, $U_r^{0,\br}(\tmt)$ is the process reflected  at $0$ of  the process $( X(\tmt) - X(\widehat{T}_{\br}(1)) +b_{1}; t \geq \widehat{T}_{\br}(1) )$. 
The process $U_r^{0,\br}$ can be constructed by repeating this procedure.
It is clear that it admits a decomposition
\begin{align*}
	U_r^{0,\br}(\tmt) = X(\tmt) - L_r^{0,\br}(\tmt) + R_r^{0,\br}(\tmt), \quad t \geq 0.
\end{align*}
Notice that for $\br$, the strategy $\pi^{0,\br} \eqdef (L_r^{0,\br}, R_r^{0,\br})$ is admissible for the  problem described at the beginning of this section, since \eqref{bailout_admissible} holds by Assumption \ref{assump_1_aux} and  Proposition \ref{p1}. We denote its expected NPV $v^{\pi^{0,\br}}_{\alpha}$, defined in \eqref{va1},  by $v^{\br}_{\alpha}$.

Since the expression of $v_{\alpha}^{\br}$ is given trough the $q$-scale functions, we shall give some preliminaries about them.  Additionally,   we shall mention about the optimal bailout dividend problems  without transaction costs when dividends are made periodical in time. This problem was studied by Noba et al. \cite{NobPerYamYan}.
\section{Preliminaries}\label{pre}

For each $q \geq 0$, there exists a function $W^{(q)}$, called the $q$-scale function  such  that $W^{(q)}(x)=0$ for $x\in(-\infty,0)$, and it is a strictly increasing function on $(0,\infty)$, which is  defined by its Laplace transform in the following way
\begin{equation} \label{scale_function_laplace}
	 \int_0^\infty  \mathrm{e}^{-\theta x} W^{(q)}(x) \der x = \frac 1 {\psi(\theta)-q}, \quad \theta > \Phi(q),
\end{equation}
where
\begin{equation}\label{r1}
	\Phi(q) \eqdef  \sup \{ \lambda \geq 0: \psi(\lambda) = q\}.
\end{equation}


We also define, for $x \in \R$,
\begin{align*}
	&\overline{W}^{(q)}(x) \eqdef   \int_0^x W^{(q)}(y) \der y, \quad
	\overline{\overline{W}}^{(q)}(x) \eqdef   \int_0^x \overline{W}^{(q)}(y) \der y, \\
	&Z^{(q)}(x) \eqdef 1 + q \overline{W}^{(q)}(x),  \quad
	\overline{Z}^{(q)}(x) \eqdef  \int_0^x Z^{(q)} (z) \der z = x + q  \int_0^x  \int_0^z W^{(q)} (w) \der w \der z.
\end{align*}
Notice that  ${\overline{W}}^{(q)}(x)=\overline{\overline{W}}^{(q)}(x) = 0$, $Z^{(q)}(x) = 1$ and $\overline{Z}^{(q)}(x) = x$ for $x \leq 0$, due to $W^{(q)}(x) = 0$ for $x\in(-\infty, 0)$.

\begin{rem}\label{remark_smoothness_zero}
\begin{enumerate}
		\item The function $W^{(q)}$ is differentiable a.e. In particular, if $X$ is of unbounded variation or the \lev measure is atomless, we have ${W^{(q)}} \in \hol^1(\R \backslash \{0\})$ (see for instance \cite[Theorem 3]{Chan2011}.  
		\item As in Lemmas 3.1 and 3.2 of \cite{KKR},
		\begin{align*} 
			\begin{split}
				{W^{(q)}} (0) &= \left\{ \begin{array}{ll} 0 & \textrm{if $X$ is of unbounded
						variation,} \\ \frac 1 {c} & \textrm{if $X$ is of bounded variation,}
				\end{array} \right. \\
				{W^{(q)\prime}} (0+) &
				=
				\left\{ \begin{array}{ll}  \frac 2 {\sigma^2} & \textrm{if }\sigma > 0, \\
					\infty & \textrm{if }\sigma = 0 \; \textrm{and} \; \Pi(0,\infty) = \infty, \\
					\frac {q + \Pi(0,\infty)} {c^2} &  \textrm{if }\sigma = 0 \; \textrm{and} \; \Pi(0,\infty) < \infty.
				\end{array} \right.
			\end{split}
		\end{align*}
		\item 
		It is well known (see \cite[p. 89]{L2009}) that the function
		\begin{equation}\label{W1}
		W_{\Phi(q)}(x) \eqdef \expo^{-\Phi(q) x}{W^{(q)}} (x)
		\end{equation}
		is log-concave on $(0,\infty)$  and that
		\begin{equation}\label{W2}
			W_{\Phi(q)}(x) \eqdef \expo^{-\Phi(q) x}{W^{(q)}} (x) \nearrow 1/\psi'(\Phi(q)),\  \text{as}\ x \uparrow \infty.
		\end{equation}
	\end{enumerate}
\end{rem}
Due to Remark \ref{remark_smoothness_zero}.1 we will make the following assumption throughout the paper.
\begin{assump}
	We assume that either $X$ has unbounded variation or $\Pi$ is absolutely continuous w.r.t. the Lebesgue measure.
\end{assump}


We also define, for $q, \theta\geq 0$ and $x \in \R$,
\begin{equation}\label{z_t}
	Z^{(q)}(x,\theta)\eqdef\expo^{\theta x}\bigg\{1+[q-\psi(\theta)] \int_{0}^{x}\expo^{-\theta z}W^{(q)}(z)\der z\bigg\}.
\end{equation}	
In particular, if $r>0$, by \eqref{scale_function_laplace}, we get
\begin{align}\label{def_z_nuevo}
	Z^{(q)}(x,\Phi(q+r)) &=\expo^{\Phi(q+r) x} \bigg[ 1 -r  \int_0^{x} \expo^{-\Phi(q+r) z} W^{(q)}(z) \der z	\bigg]\notag   \\
	&=r  \int_0^{\infty} \expo^{-\Phi(q+r) z} W^{(q)}(z+x) \der z,
\end{align}
By differentiating \eqref{def_z_nuevo} w.r.t. $x$,
\begin{align}\label{eq5.0}
	Z^{(q) \prime}(x,\Phi(q+r)) \eqdef \frac \partial {\partial x}Z^{(q)}(x,\theta)\bigg| _{\theta=\Phi(q+r)}= \Phi(q+r) Z^{(q)}(x,\Phi(q+r))	- r W^{(q)}(x), \quad x > 0. 
\end{align}
The following result will be useful for later purposes in the paper. The reader can find its proof in Appendix \ref{B}. 
\begin{lema}\label{Lh}
	Let 
	\begin{align}
		h^{(q,r)}(u)&\eqdef \frac{Z^{(q)}(u,\Phi(q+r))}{rW^{(q)}(u)}\quad \text{for} \ u\in(0,\infty),\label{ineq3}
	\end{align}
	Then, 
	$h^{(q,r)}$ is decreasing on $(0,\infty)$, and satisfies
	\begin{align}
		\lim_{u\downarrow0}h^{(q,r)}(u)=\frac{1}{rW^{(q)}(0)}\quad\text{and}\quad\lim_{u\uparrow\infty}h^{(q,r)}(u)=\frac{1}{\Phi(q+r)-\Phi(q)}.\label{e1}
	\end{align}
	
\end{lema}

Proceeding as in \cite[Eq. (6)]{lrz}, we have 
\begin{align}\label{eq1}
	\begin{split}
		W^{(q+r)}(x)-W^{(q)}(x)&=r \int_0^xW^{(q+r)}(u)W^{(q)}(x-u) \der u,\\
		Z^{(q+r)}(x)-Z^{(q)}(x)&=r \int_0^xW^{(q+r)}(u)Z^{(q)}(x-u) \der u,\\  
		Z^{(q+r)}(x,\theta)-Z^{(q)}(x,\theta)&=r \int_0^xW^{(q+r)}(u)Z^{(q)}(x-u,\theta) \der u.\ 
	\end{split}
	\qquad x \in \R.
\end{align}
 For $b\geq 0$ and $x \in \R$,  we define
\begin{align}
	\begin{split}
		W_{b}^{(q, r)} (x ) 
		&\eqdef W^{(q)}(x) + r  \int_b^{x } W^{(q + r)} (x-y)W^{(q)} (y) \der y,\\
		\overline{W}_{b}^{(q, r)} (x ) 
		&\eqdef \overline{W}^{(q)}(x) + r  \int_b^{x } W^{(q + r)} (x-y)\overline{W}^{(q)} (y) \der y,\\
		Z_{b}^{(q, r)} (x ) 
		&\eqdef Z^{(q)}(x) + r  \int_b^{x } W^{(q + r)} (x-y) Z^{(q)} (y) \der y, \\
		\overline{Z}_{b}^{(q, r)} (x ) 
		&\eqdef\overline {Z}^{(q)}(x) + r  \int_b^{x } W^{(q + r)} (x-y) \overline{Z}^{(q)} (y) \der y, \\
		Z_{b}^{(q, r)} (x,\theta ) 
		&\eqdef Z^{(q)}(x,\theta) + r  \int_b^{x } W^{(q + r)} (x-y)Z^{(q)} (y,\theta) \der y. 
	\end{split}
	\label{identities} 
\end{align}
Notice that for $x \in [0,b]$, the identities in \eqref{identities} reduce to
\begin{equation*}
	W_{b}^{(q, r)} (x ) =  W^{(q)}(x),\quad Z_{b}^{(q, r)} (x ) = Z^{(q)}(x), \quad  {Z}_{b}^{(q, r)} (x,\theta ) =  {Z}^{(q)}(x,\theta).
\end{equation*}
In addition, for $x \in \R$ we have by \eqref{eq1} that
\begin{equation*}
	W_{0}^{(q, r)} (x ) = W^{(q+r)}(x), \quad Z_{0}^{(q, r)} (x ) = Z^{(q+r)}(x).
\end{equation*}

For a comprehensive study on the scale functions and their application, see \cite{KKR,Kyp}.

\subsection{Bailout dividend problem with periodic dividend payments without transaction costs}

Considering $\alpha=0$ in \eqref{va1}, i.e. there are no transaction costs for paying out  to the shareholders, Noba et al. \cite{NobPerYamYan}	 proved that an optimal strategy  for  the problem \eqref{MAP_Value}, is of the periodic-classical $b$-barrier type.  This strategy, denoted as $\pi^{b,0}$ with a fixed barrier $b> 0$, consists of two components: $L^{b,0}_{r}\eqdef\{L^{b,0}_{r}(\tmt):\tmt\geq0\}$ and $R^{b,0}_{r}\eqdef\{R^{b,0}_{r}(\tmt):\tmt\geq0\}$, where $L^{b,0}_{r}(\tmt)$ represents the cumulative amounts of Parisian reflection until the time $\tmt$, while $R^{b,0}_{r}(\tmt)$ represents the cumulative amounts of classical reflection at zero until the time $\tmt$. The surplus (controlled) process $U^{b,0}=X-L^{b,0}_{r}+R^{b,0}_{r}$ becomes the  L\'evy process with  Parisian reflection above and classical reflection below. For more detailed information, refer to \cite{NobPerYamYan,PY2018}.

By Lemma 3.1 in \cite{NobPerYamYan}; see also Corollaries 10 and 11 in \cite{PY2018}, we get that
\begin{align}\label{v2.1}
	v^{b}_{0}(x)&=\E_{x} \left[  \int_{[0, \infty)} \expo^{-q  \tmt} \der L_r^{\pi^{b,0}}(\tmt) - \beta  \int_{[0, \infty)} \expo^{-q  \tmt} \der R_r^{\pi^{b,0}}(\tmt) \right]\notag\\
	&=- C_b\left( Z_b^{(q,r)}(x) - r Z^{(q)}(b) \overline{W}^{(q+r)}(x-b) \right) - r \overline{\overline{W}}^{(q+r)}(x-b)\notag\\
	&\quad + \beta \left[ \overline{Z}_b^{( q ,r)}(x) + \frac{\psi'(0+)}{q} - r \overline{Z}^{(q)}(b) \overline{W}^{(q+r )}(x-b) \right],\quad x\in\R,
\end{align}	
 with
\begin{equation}\label{v2.2}
	C_b \eqdef 
	\frac{g_{1}(b)}{q\Phi(q+r)},
\end{equation}
where 
\begin{equation}
	g_{1}(b)\eqdef q\beta+\frac{r[\beta Z^{(q)}(b)-1]}{Z^{(q)}(b,\Phi(q+r))}\quad\text{for}\ b>0.\label{v6}
\end{equation}
Defining $b^{*}_{r}$ as
\begin{equation}\label{b2}
	b^{*}_{r}=\inf\{b\geq0:H^{(q,r)}(b)\leq1/\beta\},
\end{equation}
where
\begin{align}\label{fpt_pr}
	H^{(q,r)}(b)\eqdef\E_{b}\left[\expo^{-q\bar{\tau}_{0}^-(r;b)}\right]
	&=Z^{(q)}(b )-q\frac{Z^{(q)}(b,\Phi(q+r))}{Z^{(q)\prime}(b,\Phi(q+r))}W^{(q)}(b),
\end{align}
Noba et al. \cite{NobPerYamYan} proved that  $\pi^{b^{*}_{r},0}=(L^{b^{*}_{r},0}_{r},R^{b^{*}_{r},0}_{r})$  is an optimal strategy for the problem \eqref{MAP_Value}, when $\alpha=0$. 
Here $\bar{\tau}_{0}^-(r;b)\eqdef\inf\{t>0: U^{b}_{r}(\tmt)<0\}$ and $U^{b}_{r}$ is the Parisian reflected process of $X$ from above at $b$ (without classical reflection).
with $\bar{\tau}_{0}^-(r;b)\eqdef\inf\{t>0: U^{b}_{r}(\tmt)<0\}$ and $U^{b}_{r}$ as the Parisian reflected process of $X$ from above at $b$;  for more details of its construction, see \cite{,NobPerYamYan, PY2018}.   Observe  that $H^{(q,r)}$ is strictly decreasing on $(0,\infty)$   satisfying
	\begin{equation}\label{remH_2}
		\lim_{b\downarrow0}H^{(q,r)}(b)=1- \frac{qW^{(q)}(0)}{\Phi(q+r)-rW^{(q)}(0)}\quad\text{and}\quad\lim_{b\uparrow\infty}H^{(q,r)}(b)=0.
	\end{equation}

\section{Expression and regularity of the expected NPV $v^{\br}_{\alpha}$ }\label{NPV1}

From now, consider  strategies of the type $\pi^{0,\br}$ which were introduced in Subsection \ref{PR}. Recall that $\br=(b_{1},b_{2})$ satisfies $0\leq b_{1}<b_{2}$, and the expected NPV $v^{\br}_{\alpha}\eqdef v^{\pi^{0,\br}}_{\alpha}$ is as in \eqref{va1}. Using  fluctuation theory, by means of $q$-scale functions, we provide  an expression for  $v^{\br}_{\alpha}$. Then, we shall study its regularity.

 Additionally, the following results are presented by extending the domain of $v^{\br}_{\alpha}$ to $\R$ by setting $v^{\br}_{\alpha}(x)=\beta x+v^{\br}_{\alpha}(0)$ for $x<0$.  This inclusion covers   the case where the process starts at a negative value and is immediately pushed up to 0.   The proof of the next   result can be found in Appendix \ref{C}.
\begin{prop}\label{p1}
If $x\in\R$ and $\br=(b_{1},b_{2})$ such that $0\leq b_{1}<b_{2}$, then
\begin{align}
	v_{\alpha}^{\br}(x)&=\frac{A^{(q,r)}_{\br}\Big[Z^{(q,r)}_{b_{2}}(x)-rZ^{(q)}(b_{1})\overline{W}^{(q+r)}(x-b_{2})\Big]}{qZ^{(q)}(b_{2},\Phi(q+r))+r[Z^{(q)}(b_{2})-Z^{(q)}(b_{1})]}-r\overline{\overline{W}}^{(q+r)}(x-b_{2})\notag\\
	&\quad+\beta\Big\{\overline{Z}^{(q,r)}_{b_{2}}(x)-r\overline{Z}^{(q)}(b_{2})\overline{W}^{(q+r)}(x-b_{2})\notag\\
	&\quad-\psi'(0+)\Big[\overline{W}^{(q,r)}_{b_{2}}(x) -r\overline{W}^{(q)}(b_{2})\overline{W}^{(q+r)}(x-b_{2})\Big]\Big\}\notag\\
	&\quad+r\overline{W}^{(q+r)}(x-b_{2})\{\beta[l^{(q)}(b_{2})-l^{(q)}(b_{1})]-[b_{2}-(b_{1}+\alpha)]\},\label{v1}
\end{align}
with
\begin{align}
l^{(q)}(x)&\eqdef\overline{Z}^{(q)}(x)-\psi'(0+)\overline{W}^{(q)}(x)=k^{(q)}(x)-\frac{\psi'(0+)}{q}Z^{(q)}(x),\label{v2}\\
A^{(q,r)}_{\br}
&\eqdef\frac{r}{\Phi(q+r)}[1-\beta Z^{(q)}(b_{2})]+r[b_{2}-b_{1}-\alpha]\notag\\
&\quad-\beta r[l^{(q)}(b_{2})-l^{(q)}(b_{1})]-\beta\Big[\frac{q}{\Phi(q+r)}-\psi'(0+)\Big]Z^{(q)}(b_{2},\Phi(q+r)).\label{v3}
\end{align}
\end{prop}	
From \eqref{v1} and \eqref{v2}, notice that
\begin{align} \label{v_pi}
	v_{\alpha}^{\br} (x)&=\frac{A^{(q,r)}_{\br}Z^{(q)}(x)}{qZ^{(q)}(b_{2},\Phi(q+r))+r[Z^{(q)}(b_{2})-Z^{(q)}(b_{1})]}+\beta l^{(q)}(x)\ \text{for}\ x\leq b_{2} .
\end{align}
Considering \eqref{v_pi} with   $x=b_{1}$,  \eqref{v1} can be rewritten in the following way
\begin{align}
	v_{\alpha}^{\br}(x)
	&=\frac{A^{(q,r)}_{\br}Z^{(q,r)}_{b_{2}}(x)}{qZ^{(q)}(b_{2},\Phi(q+r))+r[Z^{(q)}(b_{2})-Z^{(q)}(b_{1})]}-r\overline{\overline{W}}^{(q+r)}(x-b_{2})\notag\\
	&\quad+\beta\Big[\overline{Z}^{(q,r)}_{b_{2}}(x)-\psi'(0+)\overline{W}^{(q,r)}_{b_{2}}(x)\Big] -r\overline{W}^{(q+r)}(x-b_{2})\big[v^{\br}_{\alpha}(b_{1})+b_{2}-(b_{1}+\alpha)\big]\quad\text{for}\ x\in\R.\label{v1.0}
\end{align}
Meanwhile,  taking first derivatives w.r.t. $x$ in \eqref{identities}, we see that 
\begin{align}\label{iden2}
	\begin{split}
	\frac{\der}{\der x}\overline{W}^{(q,r)}_{b}(x)
	&=W^{(q,r)}_{b}(x) + rW^{(q + r)} (x-b)\overline{W}^{(q)} (b)\\
	\frac{\der}{\der x}\overline{Z}_{b}^{(q, r)} (x ) 
	&={Z}^{(q,r)}_{b}(x) + rW^{(q + r)} (x-b)\overline{Z}^{(q)} (b), \\
	\frac{\der}{\der x}{Z}_{b}^{(q, r)} (x ) 
	&=q{W}^{(q,r)}_{b}(x) + rW^{(q + r)} (x-b){Z}^{(q)} (b).
	\end{split}
\end{align}	
Then, taking first and second derivatives in \eqref{v1.0}, by \eqref{v_pi} and  \eqref{iden2}, it can be checked that  for $x>0$,
\begin{align}
	v_{\alpha}^{\br\,\prime}(x)
	&=\frac{qA^{(q,r)}_{\br}{W}^{(q,r)}_{b_{2}}(x) }{qZ^{(q)}(b_{2},\Phi(q+r))+r[Z^{(q)}(b_{2})-Z^{(q)}(b_{1})]}-r{\overline{W}}^{(q+r)}(x-b_{2})\notag\\
	&\quad+\beta\Big\{{Z}^{(q,r)}_{b_{2}}(x)-\psi'(0+){W}^{(q,r)}_{b_{2}}(x)\Big\}\notag\\
	&\quad+rW^{(q + r)} (x-b_{2})\big\{v^{\br}_{\alpha}(b_{2})-v^{\br}_{\alpha}(b_{1})-(b_{2}-b_{1}-\alpha)\big\},\label{v4}\\
	v_{\alpha}^{\br\,\prime\prime}(x)&=\frac{qA^{(q,r)}_{\br}\bigg[W^{(q)\prime}(x)+rW^{(q+r)}(x-b_{2})W^{(q)}(b_{2})+r  \int_{b_{2}}^{x}W^{(q+r)}(x-y)W^{(q)\prime}(y)\der y\bigg] }{qZ^{(q)}(b_{2},\Phi(q+r))+r[Z^{(q)}(b_{2})-Z^{(q)}(b_{1})]}\notag\\
	&\quad+\beta\bigg\{q{W}^{(q,r)}_{b_{2}}(x) + rW^{(q + r)} (x-b_{2}){Z}^{(q)} (b_{2})\notag\\
	&\quad-\psi'(0+)\bigg[W^{(q)\prime}(x)+rW^{(q+r)}(x-b_{2})W^{(q)}(b_{2})+r  \int_{b_{2}}^{x}W^{(q+r)}(x-y)W^{(q)\prime}(y)\der y\bigg]\bigg\}\notag\\
	&\quad-rW^{(q+r)}(x-b_{2})+rW^{(q + r)\prime} (x-b_{2})\big\{v^{\br}_{\alpha}(b_{2})-v^{\br}_{\alpha}(b_{1})-(b_{2}-b_{1}-\alpha)\big\}.\label{v5}
\end{align}
From here and by the smoothness of the scale functions on $\R\setminus\{0\}$, we get the following result.
\begin{lema}\label{l1}
	Let $\br=(b_{1},b_{2})$ such that $0\leq b_{1}<b_{2}$. Then,
	\begin{enumerate}
		\item[(i)] If $X$ is of bounded variation, $v_{\alpha}^{\br}\in \hol^{0}(\R)\cap\hol^{1}(\R\setminus\{0,b_{2}\})$.
		\item[(ii)] If $X$ is of unbounded variation, $v_{\alpha}^{\br}\in \hol^{1}(\R)\cap\hol^{2}(\R\setminus\{0,b_{2}\})$.
	\end{enumerate}	
\end{lema}
\subsection{Smoothness condition}
By \eqref{v4}--\eqref{v5} and by  Lemma \ref{l1}, we get that if $\br=(b_{1},b_{2})$, with $0\leq b_{1}<b_{2}$,  is where $v^{\br}_{\alpha}$  fulfils 
\begin{align}
	v^{\br}_{\alpha}(b_{2})=v^{\br}_{\alpha}(b_{1})+b_{2}-b_{1}-\alpha,\label{eq19}
	\end{align}
it follows immediately that  $v_{\alpha}^{\br}\in \hol^{0}(\R)\cap\hol^{1}(\R\setminus\{0\})$ if $X$ is of bounded variation; and  $v_{\alpha}^{\br}\in \hol^{1}(\R)\cap\hol^{2}(\R\setminus\{0\})$ if $X$ is of unbounded variation. The following result guarantees the existence of $\br$'s where \eqref{eq19} holds. 
\begin{lema}\label{s1}
	For each $b_{1}\geq0$ there exists a unique $b_{2}>b_{1}$ where \eqref{eq19} holds. Furthermore, $b_{2}\in(b^{*}_{r},\infty)$, with $b^{*}_{r}$ as in \eqref{b2}, and
	\begin{equation}\label{eq.19.1}
		v^{\br}_{\alpha}(x)=v^{b_{2}}_{0}(x),\quad\text{for}\ x\in\R,
	\end{equation}
where $v^{b_{2}}_{0}(x)$ is as in \eqref{v2.1}.
\end{lema}
Applying \eqref{v3}--\eqref{v_pi}  in \eqref{eq19}, we see that the first statement  in the lemma above is equivalent to prove that for each $b_{1}\geq0$ there exists a unique point $b_{2}>b_{1}$ such that 
\begin{align}\label{eq20}
	g_{1}(b_{2})&=q\Phi(q+r)g_{2}(b_{2};b_{1}),
\end{align}
where $g_{1}$ is as in \eqref{v6} and
\begin{align}
	g_{2}(u;b_{1})&\eqdef\frac{\beta \Big[\overline{Z}^{(q)}(u)-\overline{Z}^{(q)}(b_{1})\Big]-[u-b_{1}-\alpha]}{[Z^{(q)}(u)-Z^{(q)}(b_{1})]}, \quad \text{for}\ u> b_{1}.\label{v7}
\end{align}

 Let us start by providing and proving some properties of  $g_{1}$ and $g_{2}$, which will assist us in  the proof of Lemma \ref{s1}. For it, let us define first the function $\xi$ as
 \begin{align}
 \xi (u)\eqdef\frac{\beta Z^{(q)}(u)-1}{qW^{(q)}(u)}\quad \text{for}\ u>0,\label{zeta}
 \end{align} 
which satisfies 
 \begin{equation}\label{L_xi}
 	\lim_{u\rightarrow0}\xi (u)=\frac{\beta-1}{qW^{(q)}(0)}\ \quad\text{and}\quad\lim_{u\rightarrow\infty}\xi (u)=\frac{\beta}{\Phi(q)}.  
 \end{equation}
 Where in the case $X$ is of unbounded variation the first equality is understood to be $\infty$.
 As it was seen in \cite{JMP2019},  $\xi$ attains its unique minimum at $a^{*}$, with
 \begin{equation}\label{a1}
 	a^{*}\eqdef \sup\{u\geq0: \xi (u)\leq\xi (x), \ \text{for all}\ x\geq0\},
 \end{equation}
 since
 \begin{align}
 	\frac{\der \xi (u)}{\der u}&=\frac{W^{(q)\prime}(u)}{W^{(q)}(u)}\bigg[\beta\frac{W^{(q)}(u)}{W^{(q)\prime}(u)}-\xi(u)\bigg],\label{derzeta}
 \end{align}
  and 
 \begin{align*}
 	&\beta\frac{W^{(q)}(u)}{W^{(q)\prime}(u)}<\xi(u)\ \text{for}\ u\in(0,a^{*}),\quad \beta\frac{W^{(q)}(u)}{W^{(q)\prime}(u)}=\xi(u)\ \text{at}\ u=a^{*},\\
 	&\text{and}\quad \beta\frac{W^{(q)}(u)}{W^{(q)\prime}(u)}>\xi(u)\ \text{for}\ u\in(a^{*},\infty).
 \end{align*}
 Recall that $W^{(q+r)}/W^{(q+r)\prime}$ is increasing on $(0,\infty)$ (because of identities (8.23) and (8.26) in \cite{Kyp}). 
 Furthermore,  defining the function $H_{1}$ as
	\begin{align}
		H_{1}(u)&\eqdef\E[\expo^{-q\hat{\tau}_u}]=Z^{(q)}(u)-q\frac{[W^{(q)}(u)]^{2}}{W^{(q)\prime}(u)}\label{defH}\quad \text{for}\ u\in(0,\infty),
	\end{align}  
	where  $\hat{\tau}_a$, with $a>0$,  is the first entrance time of $\widehat{Y}\eqdef \overline{S}-X$ into $(a,\infty)$  and  $\overline{S}(\tmt)\eqdef \sup_{0\leq \tms\leq \tmt}(X(\tms)\vee0)$, observe  that \eqref{derzeta} can be rewritten as follows 
	\begin{align*}
		\frac{\der \xi (u)}{\der u}=-\frac{\beta W^{(q)\prime}(u)}{q[W^{(q)}(u)]^2}( H_{1}(u)-1/\beta),
	\end{align*}
	implying that $a^{*}$, given in \eqref{a1}, can  also be described in the following way
	\begin{equation}\label{a2}
		a^{*}=\inf\{a\geq0:H_{1}(a)\leq1/\beta\}.
	\end{equation}	
	The function  $H_{1}$ is decreasing and satisfies
	\begin{equation}\label{remH}
		\lim_{u\downarrow0}H_{1}(u)=1-\frac{q[W^{(q)}(0)]^{2}}{W^{(q)\prime}(0+)}\quad\text{and}\quad\lim_{u\uparrow\infty}H_{1}(u)=0.
	\end{equation} 
	 Let us recall here that $a^{*}$ is  the optimal threshold for dividend payments for the bail out dividend problem made continuously in time studied by Avram et al. \cite{AvrPalPis}; also see \cite{JMP2019}.

	 The following lemma gives a relation between $b^{*}_{r}$ and $a^{*}$  as detailed in Corollary \ref{coro1}. The proof of Lemma \ref{L_H} can be found in Appendix \ref{B}.
	\begin{lema}\label{L_H}
		Let $H_{1}$ and $H^{(q,r)}$ be as in \eqref{defH} and \eqref{fpt_pr}, respectively.  If $r_{1}<r_{2}$,  then
		\begin{equation}	\label{ineq1}
			H^{(q,r_{1})}(u)<H^{(q,r_{2})}(u)\quad \text{for}\  u\in(0,\infty).
		\end{equation}
		Moreover,  for each $r>0$, 
		\begin{equation}\label{ineq2}
			H^{(q,r)}(u)< H_{1}(u)\quad \text{for}\  x\in(0,\infty). 
		\end{equation}
	\end{lema}
	From \eqref{b2}, \eqref{remH_2}, \eqref{a2},  \eqref{remH}, and Lemma \ref{L_H}, we get the following  result.
	\begin{coro}\label{coro1}
		\begin{enumerate}
			\item[(1)] If $X$ is of unbounded variation, then for each $\beta>1$,  $0<b^{*}_{r}<a^{*}$.
			\item[(2)] If $X$ is of bounded variation,  we get that
			\begin{align*}
				H^{(q,r)}(0+)=1-\frac{q}{c\Phi(q+r)-r}< H_{1}(0+)= 1-\frac{q}{\Pi(0,\infty)+q},
			\end{align*}		
		where the last term is understood as $1$ when $\Pi(0,\infty)=\infty$.
		
			Additionally,
			\begin{enumerate}
				\item[(i)] if $\beta>1$ is such that $1/\beta<H^{(q,r)}(0+)$,  $0<b^{*}_{r}<a^{*}$;
				\item[(ii)] if $\beta>1$ is such that  $1/\beta\in [H^{(q,r)}(0+),H_{1}(0+))$,  $b^{*}_{r}=0<a^{*}$;
				\item[(iii)]  otherwise  $b^{*}_{r}=a^{*}=0$.
			\end{enumerate}
		\end{enumerate}
	\end{coro}

The following result provides some properties of the functions $g_{1}$ and $g_{2}$ that will be used thorugh the rest of the paper. Its proof is defered to Appendix \ref{D_nuevo}.
\begin{lema}\label{g.1.2}
	Let $g_{1},\ g_{2}$ be as in \eqref{v6} and  \eqref{v7}, respectively. Then,
	\begin{enumerate}
		\item[(i)] $g_{1}$ is strictly decreasing on  $(0,b^{*}_{r})$ and strictly increasing on $(b^{*}_{r},\infty)$, with $b^{*}_{r}\geq0$ as in   \eqref{b2}. Additionally,
			\begin{equation*}
				\lim_{u\downarrow0}g_{1}(u)=q\beta+ r(\beta-1)\quad\text{and}\quad	\lim_{u\uparrow\infty}g_{1}(u)=q\beta\frac{\Phi(q+r)}{\Phi(q)};
		\end{equation*}	
		\item[(ii)] for each $b_{1}\geq0$ fixed, there exists a unique $u^{*}_{b_{1}}>a^{*}$, with $a^{*}\geq0$ as in  \eqref{a2}, where $g_{2}(\cdot; b_{1})$ is strictly decreasing on  $(0,u^{*}_{b_{1}})$ and strictly increasing on $(u^{*}_{b_{1}},\infty)$. Additionally,
		\begin{align*}
			&\lim_{u\downarrow b_{1}}g_{2}(u;b_{1})=\infty\quad\text{and}\quad\lim_{u\uparrow\infty}g_{2}(u;b_{1})=\frac{\beta}{\Phi(q)}.
		\end{align*}
	\end{enumerate}
\end{lema}	

\begin{proof}[Proof of Lemma \ref{s1}]
	From \eqref{v6}, \eqref{v7} and \eqref{zeta}, notice that $g'_{1}$ and $g'_{2}(\cdot;b_{1})$ can be written  as follows
	\begin{align}
		g'_{1}(u)&=\frac{rW^{(q)}(u)}{Z^{(q)}(u,\Phi(q+r))}[g_{1}(u)-q\Phi(q+r)\xi(u)],\label{e3}\\
		g'_{2}(u;b_{1})&=-\frac{qW^{(q)}(u)}{Z^{(q)}(u)-Z^{(q)}(b_{1})}[g_{2}(u;b_{1})-\xi(u)]\label{e3.0}.
	\end{align}
	Then, by Corollary \ref{coro1}, Lemma \ref{g.1.2} and  \eqref{e3}--\eqref{e3.0}, it follows that 
	\begin{align}\label{i2}
		\begin{split}
			g_{1}(u)&>q\Phi(q+r)\xi(u)\quad \text{for}\ u\in(b^{*}_{r},\infty),\\
			q\Phi(q+r)\xi(u)&>q\Phi(q+r)g_{2}(u;b_{1})\quad \text{for}\ u\in(u^{*}_{b_{1}},\infty).
		\end{split}
	\end{align}	
	Furthermore, notice that  $b^{*}_{r}\leq a^{*}<u^{*}_{b_{1}}$, due to   Corollary \ref{coro1} and Lemma \ref{g.1.2}. Then, since $b^{*}_{r}$ and $u^{*}_{b_{1}}$ are the unique points where $g_{1}$ and $g_{2}(\cdot;b_{1})$ attain their minimum, respectively,
	we conclude that  there exists a unique $b_{2}\in(b^{*}_{r},u^{*}_{b_{1}})$ such that 
	\begin{align*}
		&g_{1}(u)<q\Phi(q+r)g_{2}(u;b_{1})\ \text{for}\ u\in (0,b_{2}),\\
		&g_{1}(u)>q\Phi(q+r)g_{2}(u;b_{1})\ \text{for}\ u\in(b_{2},\infty).
	\end{align*}
	To finalize, let us check that \eqref{eq.19.1} is true. By \eqref{v_pi} and  \eqref{eq19}, it can be checked that 
	\begin{align}\label{eq19.1}
		&\frac{A^{(q,r)}_{\br}Z^{(q)}(b_{2})}{qZ^{(q)}(b_{2},\Phi(q+r))+r[Z^{(q)}(b_{2})-Z^{(q)}(b_{1})]}\notag\\
		&\qquad=\frac{A^{(q,r)}_{\br}Z^{(q)}(b_{1})}{qZ^{(q)}(b_{2},\Phi(q+r))+r[Z^{(q)}(b_{2})-Z^{(q)}(b_{1})]}+b_{2}-b_{1}-\alpha-\beta [l^{(q)}(b_{2})-l^{(q)}(b_{1})].
	\end{align}
	Applying  \eqref{eq19.1} in \eqref{v1}, it follows that
	\begin{align*} 
		v_{\alpha}^{\br} (x)&=\frac{A^{(q,r)}_{\br}\Big[Z^{(q,r)}_{b_{2}}(x)-rZ^{(q)}(b_{2})\overline{W}^{(q+r)}(x-b_{2})\Big]}{qZ^{(q)}(b_{2},\Phi(q+r))+r[Z^{(q)}(b_{2})-Z^{(q)}(b_{1})]}-r\overline{\overline{W}}^{(q+r)}(x-b_{2})\notag\\
		&\quad+\beta\Big\{\overline{Z}^{(q,r)}_{b_{2}}(x)-r\overline{Z}^{(q)}(b_{2})\overline{W}^{(q+r)}(x-b_{2})\notag\\
		&\quad-\psi'(0+)[\overline{W}^{(q,r)}_{b_{2}}(x) -r\overline{W}^{(q)}(b_{2})\overline{W}^{(q+r)}(x-b_{2})]\Big\}\quad\text{for}\ x\in\R.
	\end{align*}
	Moreover, by \eqref{v2}, \eqref{eq20} and \eqref{eq19.1}, it can be verified easily that
	\begin{align}
		&\frac{A^{(q,r)}_{\br}}{qZ^{(q)}(b_{2},\Phi(q+r))+r[Z^{(q)}(b_{2})-Z^{(q)}(b_{1})]}\notag\\	&=\frac{\beta\psi'(0+)}{q}-\frac{g_{1}(b_{2})}{q\Phi(q+r)}\label{v10}\\
		&=\frac{r[1-\beta Z^{(q)}(b_{2})]}{q\Phi(q+r)Z^{(q)}(b_{2},\Phi(q+r))}-\frac{\beta}{q}\bigg[\frac{q}{\Phi(q+r)}-\psi'(0+)\bigg].\notag
	\end{align}
	Then,
	\begin{align}	
		v_{\alpha}^{\br} (x)
		&=\bigg[\frac{r[1-\beta Z^{(q)}(b_{2})]}{q\Phi(q+r)Z^{(q)}(b_{2},\Phi(q+r))}-\frac{\beta}{\Phi(q+r)}\bigg]\Big[Z^{(q,r)}_{b_{2}}(x)-rZ^{(q)}(b_{2})\overline{W}^{(q+r)}(x-b_{2})\Big]\notag\\
		&\quad-r\overline{\overline{W}}^{(q+r)}(x-b_{2})+\beta\Big\{\frac{\psi'(0+)}{q}\Big[Z^{(q,r)}_{b_{2}}(x)-rZ^{(q)}(b_{2})\overline{W}^{(q+r)}(x-b_{2})\Big]\notag\\
		&\quad+\overline{Z}^{(q,r)}_{b_{2}}(x)-r\overline{Z}^{(q)}(b_{2})\overline{W}^{(q+r)}(x-b_{2})\notag\\
		&\quad-\psi'(0+)[\overline{W}^{(q,r)}_{b_{2}}(x) -r\overline{W}^{(q)}(b_{2})\overline{W}^{(q+r)}(x-b_{2})]\Big\}.\label{v8}
	\end{align}
	By integration by parts, notice that 
		\begin{equation*}
			\int_{b_{2}}^{x}W^{(q+r)}(x-y)Z^{(q)}(y)\der y=\overline{W}^{(q+r)}(x-b_{2})Z^{(q)}(b_{2})+q\int_{b_{2}}^{x}\overline{W}^{(q+r)}(x-y)W^{(q)}(y)\der y.
		\end{equation*}
		From here and using  \eqref{identities} together with the definition of $Z^{(q)}$, we get that 
	\begin{align}
		&\frac{\psi'(0+)}{q}\Big[Z^{(q,r)}_{b_{2}}(x)-rZ^{(q)}(b_{2})\overline{W}^{(q+r)}(x-b_{2})\Big]\notag\\
		&\quad+\overline{Z}^{(q,r)}_{b_{2}}(x)-r\overline{Z}^{(q)}(b_{2})\overline{W}^{(q+r)}(x-b_{2})-\psi'(0+)[\overline{W}^{(q,r)}_{b_{2}}(x) -r\overline{W}^{(q)}(b_{2})\overline{W}^{(q+r)}(x-b_{2})]\notag\\
		&=\frac{\psi'(0+)}{q}\bigg[Z^{(q)}(x)+qr \int_{b_{2}}^{x}\overline{W}^{(q+r)}(x-y){W}^{(q)}(y)\der y\bigg]\notag\\
		&\quad-\psi'(0+)\bigg[\overline{W}^{(q)}(x) +r \int_{b_{2}}^{x}\overline{W}^{(q+r)}(x-y){W}^{(q)}(y)\der y\bigg]+\overline{Z}^{(q,r)}_{b_{2}}(x)-r\overline{Z}^{(q)}(b_{2})\overline{W}^{(q+r)}(x-b_{2})\notag\\
		&=\overline{Z}^{(q,r)}_{b_{2}}(x)-r\overline{Z}^{(q)}(b_{2})\overline{W}^{(q+r)}(x-b_{2})+\frac{\psi'(0+)}{q}.\label{v9}
	\end{align}
	Applying \eqref{v9} in \eqref{v8} and taking into account \eqref{v2.1}--\eqref{v2.2}, we conclude that  \eqref{eq.19.1} is true.
\end{proof}

\section{Selection of  a candidate optimal strategy }\label{select1}

 To determine a candidate for an optimal strategy, we focus on the periodic-classical $\br$-barriers $\pi^{0,\br}$ (defined in Subsection \ref{PR}), specifically those strategies where the barriers of the form $\br=(b_{1},b_{2})$ belong to the following set
\begin{equation*}
	\mathcal{C}\eqdef\{(b_{1},b_{2}): 0\leq b_{1}<b_{2}\  \text{and}\ v^{\br}_{\alpha}(b_{2})=v^{\br}_{\alpha}(b_{1})+b_{2}-b_{1}-\alpha,\ \text{with}\ \br=(b_{1},b_{2})\}.
\end{equation*}
It is known that $\mathcal{C}\neq\emptyset$ because of Lemma \ref{s1}. Additionally, $\mathcal{C}$ is a trajectory set where $b_{2}$ depends on $b_{1}$.  Observe that $b_{1}\mapsto b_{2}(b_{1})$,   with the property that   \eqref{eq19} holds at   $\br_{b_{1}}\eqdef(b_{1},b_{2}(b_{1}))$,   is a continuous trajectory. 
The candidate $\br^{*}=(b_{1}^{*},b^{*}_{2})\in\mathcal{C}$ must satisfy the following condition
\begin{equation*}
	\begin{cases}
 	v^{\br^{*}\prime}_{\alpha}(b^{*}_{1})=1 &\text{if}\ b^{*}_{1}>0, \\
 	b^{*}_{1}=0& \text{otherwise.} 
 	\end{cases}
 \end{equation*}
Before stating the main result of this section (see Lemma \ref{beh2}), let us first analyse the behaviour of $b_{1}\mapsto b_{2}(b_{1})$ with $ \br_{b_1}=(b_{1},b_{2}(b_{1}))$ satisfying  \eqref{eq19}. 

Considering  $F:\mathcal{D}\longrightarrow\R$, with $\mathcal{D}\eqdef\{(b_{1},b_{2}):b_{2}>b_{1}\geq0\}$, as
\begin{equation*}
	F(b_{1},b_{2})\eqdef g_{1}(b_{2})-q\Phi(q+r)g_{2}(b_{2};b_{1})\ \text{for} \ (b_{1},b_{2})\in\mathcal{D},
\end{equation*}	
we see that  $F\in\hol^{1}(\mathcal{D})$. Additionally, from the proof of Lemma \ref{s1}, it is  known that for each $b_{1}\geq0$ there exists a unique $b_{2}(b_{1})\in(b^{*}_{r}, u^{*}_{b_{1}})$ such that  $F(b_{1},b_{2}(b_{1}))=0$.  By \eqref{eq20}, \eqref{v7}, \eqref{e3.0},  and taking into account that $\frac{\der}{\der b_{1}}F(b_{1},b_{2}(b_{1}))=0$, it can verified that 
\begin{align*}
	b'_{2}(b_{1})[g'_{1}(b_{2}(b_{1}))-q\Phi(q+r)g'_{2}(b_{2}(b_{1});b_{1})]=-\frac{q^{2}\Phi(q+r)W^{(q)}(b_{1})}{Z^{(q)}(b_{2}(b_{1}))-Z^{(q)}(b_{1})}[\xi(b_{1})-\tilde{g}(b_{1})],
\end{align*}
with $b'_{2}\eqdef \frac{\der }{\der b_{1}}b_{2}$ and $b_{1}\mapsto\tilde{g}(b_{1})\eqdef\frac{g_{1}(b_{2}(b_{1}))}{q\Phi(q+r)}$. On the other hand, using \eqref{eq20} and \eqref{e3}--\eqref{i2}, it follows that
\begin{align}\label{v11}
&g'_{1}(b_{2}(b_{1}))-q\Phi(q+r)g'_{2}(b_{2}(b_{1});b_{1})\notag\\
&\qquad\qquad=\bigg[\frac{r}{Z^{(q)}(b_{2}(b_{1}),\Phi(q+r))}+\frac{q}{Z^{(q)}(b_{2}(b_{1}))-Z^{(q)}(b_{1})}\bigg]\notag\\
&\qquad\qquad\quad\times W^{(q)}(b_{2}(b_{1}))[g_{1}(b_{2}(b_{1}))-q\Phi(q+r)\xi(b_{2}(b_{1}))]>0.
\end{align}
Then,
\begin{align}\label{v12}
	b'_{2}(b_{1})=-\frac{q^{2}\Phi(q+r)W^{(q)}(b_{1})[\xi(b_{1})-\tilde{g}(b_{1})]}{[g'_{1}(b_{2}(b_{1}))-q\Phi(q+r)g'_{2}(b_{2}(b_{1});b_{1})][Z^{(q)}(b_{2}(b_{1}))-Z^{(q)}(b_{1})]}.
\end{align}
\begin{lema}\label{beh1}
	Let $b_{1}\mapsto b_{2}(b_{1})$, with $b_{1}\geq0$, be the trajectory  such that $(b_{1},b_{2}(b_{1}))$ satisfies \eqref{eq19}. Then,
	\begin{enumerate}
		\item[(i)] If   $\xi(0+)>\tilde{g}(0+)$, then  there exists a unique $\hat{b}_{1}\in(0,b^{*}_{r})$ where $b_{1}\mapsto b_{2}(b_{1})$ is strictly decreasing on $(0,\hat{b}_{1})$, and strictly  increasing on $(\hat{b}_{1},\infty)$.
		\item[(ii)]  If   $\xi(0+)<\tilde{g}(0+)$, $b_{1}\mapsto b_{2}(b_{1})$  is strictly increasing on $(0,\infty)$.
	\end{enumerate}
\end{lema}

\begin{proof}
	To verify the results above, it is sufficient to study the trajectory $b_{1}\mapsto \hat{g}(b_{1})\eqdef \xi(b_{1})-\tilde{g}(b_{1})$, with $b_{1}\geq0$, due to \eqref{v11}--\eqref{v12}.  If $b_{1}\geq b^{*}_{r}$, using \eqref{i2} and since $g_{1}$ is strictly increasing on $(b^{*}_{r},\infty)$, it follows that 
	\begin{equation}\label{v14.0}
		\tilde{g}(b_{1})=\frac{g_{1}(b_{2}(b_{1}))}{q\Phi(q+r)}>\frac{g_{1}(b_{1})}{q\Phi(q+r)}\geq\xi(b_{1}),
	\end{equation}
	because of $b_{2}(b_{1})>b_{1}$.  (i) if $\xi(0+)>\tilde{g}(0+)$, we have that  $b^{*}_{r}>0$ and $\hat{g}(0+)>0$. Thus, there exists $\bar{b}_{1}\in(0,b^{*}_{r})$ such that  
	\begin{equation*}
	\hat{g}(b_{1})<0\quad \text{for} \  b_{1}\in(\bar{b}_{1},\infty).  
	\end{equation*}
	Then, there exists $ \tilde{b}_{1}\in(0,\bar{b}_{1})$ such that $\hat{g}(\tilde{b}_{1})=0$.  Defining 
	\begin{equation}\label{v13.0}
	\hat{b}_{1}=\inf\{\bar{b}_{1}\in(0,b^{*}_{r}): \tilde{g}(b_{1})>\xi(b_{1})\ \text{for}\ b_{1}\in(\bar {b}_{1},\infty)\},
	\end{equation}
	 we take $\underbar{$b$}_{1}\eqdef\inf\{b_{1}\in(0,\hat{b}_{1}):\tilde{g}(b_{1})=\xi(b_{1})\}$. By  the previous discussion  and  by  \eqref{v12}, we get that $b_{1}\mapsto b_{2}(b_{1})$ is strictly decreasing $(0,\underbar{$b$}_{1})$ and strictly increasing on $(\hat{b}_{1},\infty)$.     On the other hand, note that 
	\begin{align}\label{v13}
	\tilde{g}'(b_{1})= b'_{2}(b_{1})\frac{g'_{1}(b_2(b_{1}))}{q\Phi(q+r)}.
	\end{align}
	From here and by the monotonicity of $b_{2}(\cdot)$ on $(0,\underbar{$b$}_{1})\cup(\hat{b}_{1},\infty)$ we get   that $\tilde{g}$ is  also strictly decreasing on  $(0,\underbar{$b$}_{1})$ and   strictly increasing on  $(\hat{b}_{1},\infty)$, due to  $b_2(b_1)>b_r^*$ (which holds by Lemma \ref{s1}) and $g'_{1}(b_{2}(b_{1}))>0$   for $b_{1}\geq0$. Using \eqref{v13}, we shall show by contradiction that  $\underbar{$b$}_{1}=\hat{b}_{1}$. Assume that  $\underbar{$b$}_{1}<\hat{b}_{1}$. Define $\mathcal{G}$ as the collection of points $b_{1}\in[\underbar{$b$}_{1},\hat{b}_{1}]$ where $b_{1}\mapsto b_{2}(b_{1})$ attains a local minimum. Note that  $\{\underbar{$b$}_{1},\hat{b}_{1}\}\subset\mathcal{G}$.  Additionally, $\mathcal{G}$ does not admit open sets contained on it, because if $(b^{(1)}_{1},b^{(2)}_{1})\subset\mathcal{G}$ for some $b^{(1)}_{1},b^{(2)}_{1}\in (\underbar{$b$}_{1},\hat{b}_{1})$, with $b^{(1)}_{1}<b^{(2)}_{1}<b^{*}_{r}$,  we have that $\tilde{g}=\xi$ on $(b^{(1)}_{1},b^{(2)}_{1})$, because of \eqref{v12}. Hence,   $\tilde{g}'=\xi'<0$ on $(b^{(1)}_{1},b^{(2)}_{1})$  (due to $b^{*}_{r}\leq a^{*}$),  
	which is a contradiction because $\tilde{g}'=0$ holds  on $(b^{(1)}_{1},b^{(2)}_{1})$ due to \eqref{v13} and  $b'_{2}(b_{1})=0$ for $b_{1}\in(b^{(1)}_{1},b^{(2)}_{1})$.  Without loss of generalization, assume that there is a point $b^{(1)}_{1}\in(\underbar{$b$}_{1},\hat{b}_{1})$ such that $b'_{2}(b^{(1)}_{1})=0$, $b_{1}\mapsto b_{2}(b_{1})$ is strictly increasing on $(\underbar{$b$}_{1},b^{(2)}_{1})$ and  strictly decreasing on $(b^{(2)}_{1},\hat{b}_{1})$, it gives that
	\begin{align*}
		&\xi(b_{1})<\tilde{g}(b_{1})\ \text{for}\ b_{1}\in(\underbar{$b$}_{1},b^{(1)}_{1}),\ \xi(b_{1})>\tilde{g}(b_{1})\ \text{for}\ b_{1}\in(b^{(1)}_{1},\hat{b}_{1})\notag\\
		 &\text{and}\ \xi(b_{1})=\tilde{g}(b_{1})\ \text{for}\ b_{1}\in\{\underbar{$b$}_{1},b^{(1)}_{1},\hat{b}_{1}\}.
	\end{align*}
	Then, $\tilde{g}'(b^{(1)}_{1})\leq\xi'(b^{(1)}_{1})<0$, which is a contradiction due to $\tilde{g}'(b^{(1)}_{1})=0$.  	From here, we conclude that the statement that appears in Lemma \ref{beh1}(i) is true.  (ii) In case that $\xi(0+)<\tilde{g}(0+)$, notice that if $b^{*}_{r}=0$, by  \eqref{v14.0}, it follows immediately that $b_{1}\mapsto b_{2}(b_{1})$ is strictly increasing on $(0,\infty)$.  If $b^{*}_{r}>0$, observe that $\hat{g}<0$ on $(\hat{b}_{1},\infty)$ where $\hat{b}_{1}$ is as in \eqref{v13.0}. If  $\hat{b}_{1}>0$, by the  continuity of $\xi$ and $\tilde{g}$, there exists a $\bar{b}_{1}\in(0,\hat{b}_{1})$ such that  $\hat{g}<0$ on $(0,\bar{b}_{1})$, which implies that $b_{1}\mapsto b_{2}(b_{1})$ and $\tilde{g}$ are strictly increasing on $(0,\bar{b}_{1})\cup(\hat{b}_{1}, \infty)$ because of \eqref{v12} and \eqref{v13}, respectively. Since $\bar{b}_{1}<\hat{b}_{1}$ and arguing in a similar way that in the case $\xi(0+)>\tilde{g}(0+)$, we conclude that $\hat{b}_{1}=0$ and therefore $b_{1}\mapsto b_{2}(b_{1})$ is strictly increasing on $(0,\infty)$.
\end{proof}	

\begin{rem}\label{beh1.0}
	\begin{enumerate}
		\item[(i)] In the case that  $X$ is of unbounded variation, it is known that  $\tilde{g}(0+)<\xi(0+)=\infty$.  Moreover $b_{2}(0)<\infty$, due to Lemma \ref{s1}. Then,  by Lemma \ref{beh1}(i), it follows that $\hat{b}_{1}>0$.
		\item[(ii)]   When $X$ is of bounded variation such that $\frac{1}{\beta}\in[H_{2}^{(q,r)}(0+),1)$,  by Corollary \ref{coro1} (ii) and (iii), it holds that $g_{1}$ is increasing on $(0, \infty)$ satisfying $g_{1}>q\Phi(q+r)\xi$  on $(0,\infty)$, due to \eqref{e3}.  This  implies $\xi(0+)<\frac{g_{1}(0+)}{q\Phi(q+r)}<\frac{g_{1}(b_{2}(0+))}{q\Phi(q+r)}=\tilde{g}(0+)$. By Lemma \ref{beh1} (ii), we get that  $\hat{b}_{1}=0$. 
		\item[(iii)] If $X$ is of bounded variation such that  $\frac{1}{\beta}\in(0,H_{2}^{(q,r)}(0+))$, we get that 
		\begin{equation*}
			\hat{b}_{1}
			\begin{cases}
				>0&\text{if}\ \xi(0+)>\tilde{g}(0+),\\
				=0&\text{if}\ \xi(0+)<\tilde{g}(0+).
			\end{cases}	
		\end{equation*}
	However, in case that $\xi(0+)=\tilde{g}(0+)$, we obtain that $\hat{b}_{1}=0$ due to $\xi'(0+)<0=\tilde{g}'(0+)$.
	\end{enumerate}
\end{rem}

\subsection{The candidate optimal strategy}
 If $\br=(b_{1},b_{2})\in\mathcal{C}$, by  \eqref{v4}, \eqref{eq19}, \eqref{zeta} and \eqref{v10} 
 , we see that
 \begin{equation*}
 	v^{\br\,\prime}_{\alpha}(b_{1})=qW^{(q)}(b_{1})\bigg[\xi(b_{1})-\frac{g_{1}(b_{2})}{q\Phi(q+r)}\bigg]+1.
 \end{equation*}	
 Then, from here, Remark \ref{beh1.0} and  by the arguments in the proof of Lemma \ref{beh1}, we obtain the following result.
 
 \begin{lema}\label{beh2}
 		Let $b_{1}\mapsto b_{2}(b_{1})$ be the trajectory  such that ${\br}_{b_1}=(b_{1},b_{2}(b_{1}))$ satisfies \eqref{eq19} for $b_{1}\geq0$. Then, $b_{1}\mapsto v^{\br_{b_1}\prime}_{\alpha}(b_{1})=qW^{(q)}(b_{1})[\xi(b_{1})-\tilde{g}(b_{1})]+1$, for $b_{1}\geq0$, fulfils the following.
 	\begin{enumerate}
 		\item[(i)] There exists a unique $\hat{b}_{1}\in(0,b^{*}_{r})$ where 
 		\begin{equation}\label{beh3}
 			v^{\br_{b_{1}}\prime}_{\alpha}(b_{1})>1\ \text{for}\  b_{1}\in(0,\hat{b}_{1}),\  v^{\br_{b_{1}}\prime}_{\alpha}(b_{1})=1\ \text{at}\ b_{1}=\hat{b}_{1},\ v^{\br_{b_{1}}\prime}_{\alpha}(b_{1})<1\ \text{for}\ b_{1}\in(\hat{b}_{1},\infty),
 		\end{equation}	
 	only if one of the following conditions holds:
 	\begin{enumerate}
 		\item[(a)] $X$ is of unbounded variation;
 		\item[(b)] $X$ is of bounded variation, $\frac{1}{\beta}\in(0,H_{2}^{(q,r)}(0+))$ and $\xi(0+)>\tilde{g}(0+)$;
 	\end{enumerate}	
 		\item[(ii)]  
 	$v^{\br_{b_{1}}\prime}_{\alpha}(b_{1})<1$  for  $b_{1}\in(0,\infty)$, only if one of the following conditions holds:
 	\begin{enumerate}
 		\item[(a)] $X$ is  of bounded variation and $\frac{1}{\beta}\in[H_{2}^{(q,r)}(0+),1)$;
 		\item[(b)] $X$ is of bounded variation, $\frac{1}{\beta}\in(0,H_{2}^{(q,r)}(0+))$ and $\xi(0+)\leq\tilde{g}(0+)$.
 	\end{enumerate}
 	\end{enumerate}
 \end{lema}
 
\begin{rem}\label{beh4}
By the previous result, our candidate $\pi^{0,\br^{*}}$ to be an optimal strategy is such that $b^{*}_{1}=\hat{b}_{1}$ and $b^{*}_{2}=b_{2}(\hat{b}_{1})$, with $\hat{b}_{1}>0$ satisfying \eqref{beh3}, if  one of the two conditions in Lemma \ref{beh2} (i) holds. Otherwise, we set $b^{*}_{1}=0$ and $b^{*}_{2}=b_{2}(0)$. 	
	
\end{rem}

\section{Verification of optimality} \label{sec6}
In this section, we shall confirm that  $\pi^{0,\br^{*}}$ as in Remark \ref{beh4} is the optimal strategy.  

\begin{teor}\label{op1}
	The strategy   $\pi^{0,\br^{*}}$   considered  in Remark \ref{beh4} is optimal and the value function defined  in \eqref{MAP_Value} is given by $v_{\alpha}=v^{\br^{*}}_{\alpha}$. 
\end{teor}

Let $\mathcal{L}$ be the infinitesimal generator associated with the process $X$ applied to a
{$\mathcal{C}^1(0,\infty)$ (resp., $\mathcal{C}^2(0,\infty)$)} function $f$ for the case $X$ is of bounded (resp., unbounded) variation:
\begin{align*}
	\mathcal{L} f(x)&\eqdef \mu f'(x)+\frac{1}{2}\sigma^{2}f''(x)\notag\\
	&\quad+ \int_{(-\infty,0)}[f(x+z)-f(x)-f'(x)z1_{\{-1<z<0\}}]\Pi(\der z)\ \quad{for}\ x>0.
\end{align*}

Prior to presenting the proof of Theorem \ref{op1}, we introduce two crucial lemmas. Lemma \ref{prop_HJB_aux} serves as a verification lemma, with a proof akin to Proposition 5.1 in \cite{NobPerYamYan}, and is therefore omitted. Lemma \ref{P2} details the behaviour of $v_{\alpha}^{\mathbf{b}^{*}\prime}$, with $\pi^{0,\mathbf{b}^{*}}$  as outlined in Remark \ref{beh4}.

\begin{lema}[Verification lemma]\label{prop_HJB_aux}
	Suppose that $\hat{\pi} \in \A$ is such that $v^{\hat{\pi}}_{\alpha}\in\hol(\R)\cap\hol^1((0,\infty))$ (respectively, $\hol^{1}(\R)\cap\hol^2((0,\infty))$) for the case that $X$ has paths of bounded (respectively, unbounded) variation. In addition, suppose that
	\begin{equation}\label{HJB_aux}
		\begin{split}
			(\mathcal{L} - q) v^{\hat{\pi}}_{\alpha}(x) + r \max_{0 \leq l \leq x} \lbrace [l -\alpha]\uno_{\{l>0\}}+ v^{\hat{\pi}}_{\alpha}(x-l) - v^{\hat{\pi}}_{\alpha}(x) \rbrace&\leq 0, \qquad x \geq 0,\\ v^{\hat{\pi}\,\prime}_{\alpha}(x) &\leq \beta, \qquad x \geq 0, \\
			\inf_{x \geq 0} v^{\hat{\pi}}_{\alpha}(x) &> -m \text{ for some } m > 0.
		\end{split}
	\end{equation}
	Then $\hat{\pi}$ is an optimal strategy and $v^{\hat{\pi}}_{\alpha}(x) = v_{\alpha}(x)$ for all $x \geq 0$. 
\end{lema}

\begin{lema}\label{P2}
	Let   $\pi^{0,\br^{*}}$ be the strategy  provided  in Remark \ref{beh4}. Then, the following holds:
	\begin{enumerate}
		\item[(i)] For $x\in(0,{b^{*}_{1}})$, $v^{\br^{*}\prime}_{\alpha}(x)\in (1,\beta)$.
		\item[(ii)] For $x\in(b^{*}_{1},\infty)$, $v^{\br^{*}\prime}_{\alpha}(x)\in (0,1)$.
	\end{enumerate}
\end{lema}

\begin{proof}[Proof of Theorem \ref{op1}]
	In order to verify the optimality of $\pi^{0,\br^{*}}$, we proceed to show that $v^{\br^{*}}_{\alpha}$ satisfies \eqref{HJB_aux}. From Lemma \ref{P2}, it follows immediately that $v^{\br^{*}\prime}_{\alpha}\leq\beta$ on $(0,\infty)$. Moreover, considering $f(x)\eqdef x-b^{*}_{1}-\alpha+ v^{\br^{*}}_{\alpha}(b^{*}_{1}) - v^{\br^{*}}_{\alpha}(x)$ for $x\in(b^{*}_{1},\infty)$, from Lemma \ref{P2}, we deduce that $f$ is increasing on $(b_{2},\infty)$ satisfying 
	\begin{equation}\label{ine2}
		f(x)<0\ \text{for}\ x\in(b^{*}_{1},b^{*}_{2}),\quad f(x)=0\ \text{at}\ x =b^{*}_{2}\quad\text{and}\quad f(x)>0\ \text{for}\ x\in(b^{*}_{2},\infty),
	\end{equation}
	because of  $f(b^{*}_{1}+)=-\alpha$ and \eqref{eq19}. On the other hand,  defining for each $x>0$ fixed, $f_{x}(l)=[l -\alpha]\uno_{\{l>0\}}+ v^{\br^{*}}_{\alpha}(x-l) - v^{\br^{*}}_{\alpha}(x)$ for $l\in[0,x]$, observe that $f_{x}(0-)=0$, $f_{x}(0+)=-\alpha$ and  $	f'_{x}(l)=1-v^{\br^{*}\prime}_{\alpha}(x-l)$  for $l\in(0,x)$. Lemma \ref{P2} implies that
	\begin{equation}\label{ine1}
		f'_{x}(l)>0\ \text{for}\ l\in (0,x-b^{*}_{1}), \quad  f'_{x}(l)=0\ \text{at}\  l=x-b^{*}_{1},  \quad\text{and}\quad	f'_{x}<0\ \text{for}\  l\in(x-b^{*}_{1}, x).
	\end{equation}
	If $x\leq b^{*}_{1}$,  based on \eqref{ine1}, we observe that $f_{x}$ decreases on $(0,x)$. Thus, $\max_{l\in[0,x]}f_{x}(l)=0$. If $x\in(b^{*}_{1},b^{*}_{2}]$, $f_{x}$ attains its unique maximum at $l=x-b_{1}$, however, by \eqref{ine2}, $f_{x}(x-b_{1})=f(x)\leq0$. This implies that $\max_{l\in[0,x]}f_{x}(l)=0$. In the case of $x>b^{*}_{2}$, $f_{x}$ continues to reach its maximum at $l=x-b_{1}$, meanwhile, by \eqref{ine2}, we get that $f_{x}(x-b_{1})=f(x)>0$. Therefore,  $\max_{l\in[0,x]}f_{x}(l)=f(x)$. In summary, as mentioned above,
	\begin{equation}\label{ine3}
		\max_{0 \leq l \leq x} \lbrace [l -\alpha]\uno_{\{l>0\}}+ v^{\br^{*}}_{\alpha}(x-l) - v^{\br^{*}}_{\alpha}(x) \rbrace=
		\begin{cases}
			0&\text{if}\ x\in(0,b^{*}_{2}),\\
			x-b^{*}_{1}-\alpha+v^{\br^{*}}_{\alpha}(b^{*}_{1})-v^{\br^{*}}_{\alpha}(x)&\text{if}\ x\in[b^{*}_{2},\infty).
		\end{cases}	
	\end{equation}
	On the other hand,  by Lemma 5.1 in \cite{NobPerYamYan} and using \eqref{eq19}--\eqref{eq.19.1}, it gives that 
	\begin{align}\label{ine4}|
		[\mathcal{L}-q]v^{\br^{*}}_{\alpha}(x)&=[\mathcal{L}-q]v^{b^{*}_{2}}_{0}(x)\notag\\
		&=
		\begin{cases}
			0&\text{if}\ x\in(0,b^{*}_{2}),\\
			-r\Big[x-b^{*}_{2}+v^{b^{*}_{2}}_{0}(b^{*}_{2})-v^{b^{*}_{2}}_{0}(x)\Big]&\text{if}\ x\in[b^{*}_{2},\infty).
		\end{cases}	\notag\\
		&=
		\begin{cases}
			0&\text{if}\ x\in(0,b^{*}_{2}),\\
			-r\Big[x-b^{*}_{1}-\alpha+v^{\br^{*}}_{\alpha}(b^{*}_{1})-v^{\br^{*}}_{\alpha}(x)\Big]&\text{if}\ x\in[b^{*}_{2},\infty).
		\end{cases}	
	\end{align}
	Therefore,  from \eqref{ine3} and \eqref{ine4}, $(\mathcal{L} - q) v^{\br^{*}}_{\alpha}(x) + r \max_{0 \leq l \leq x} \lbrace [l -\alpha]\uno_{\{l>0\}}+ v^{\hat{\pi}}_{\alpha}(x-l) - v^{\br^{*}}_{\alpha}(x) \rbrace=0$. By the monotonicity of $v^{\br}_{\alpha}$ it follows that $\inf_{x\geq0}v^{\br^{*}}_{\alpha}(x)\geq v^{\br^{*}}_{\alpha}(0)>-\infty$.  
\end{proof}

To finalize this section, we will check the veracity of Lemma \ref{P2}. For it, let us first introduce Lemmas \ref{lem_1}, \ref{lem_2}  and Remark \ref{P2.0}, which are useful for its proof.  We defer the proof of Lemmas  \ref{lem_1} and \ref{lem_2} to Appendix \ref{D_2}.
\begin{lema}\label{lem_1}
	For $x\in(b^{*}_{2},\infty)$, consider  $\bar{h}_{1}(x;b^{*}_{r},b^{*}_{2})$ and $\bar{h}_{2}(x;b^{*}_{r},b^{*}_{2})$ defined as follows
	\begin{align}\label{vp3}
		\begin{split}
		\bar{h}_{1}(x;b^{*}_{r},b^{*}_{2})&\eqdef r \int_{b_{r}^{*}}^{b^{*}_{2}}\frac{W^{(q+r)}(x-y)}{W^{(q+r)}(x)}W^{(q)}(y)\bigg[\frac{g_{1}(b^{*}_{2})}{q\Phi(q+r)}-\xi(y)\bigg]\der y,\\
		\bar{h}_{2}(x;b^{*}_{r},b^{*}_{2})&\eqdef \bigg[ \frac{g_{1}(b^{*}_{2})-g_{1}(b^{*}_{r})}{q\Phi(q+r)}\bigg]\bigg[1-r \int_{0}^{b^{*}_{r} } \frac{W^{(q + r)} (x-y)}{W^{(q+r)}(x)}W^{(q)} (y) \der y\bigg].
		\end{split}
	\end{align}
	Then $\bar{h}_{1}(\cdot\,;b^{*}_{r},b^{*}_{2})$, $\bar{h}_{2}(\cdot\,;b^{*}_{r},b^{*}_{2})$ are increasing and decreasing on $(b^{*}_{2},\infty)$, respectively, and satisfy
	\begin{align}\label{vp3.2}
		\lim_{x\rightarrow\infty}\bar{h}_{1}(x;b^{*}_{r},b^{*}_{2})=\lim_{x\rightarrow\infty}\bar{h}_{2}(x;b^{*}_{r},b^{*}_{2}).
	\end{align}
	\end{lema}

\begin{rem}\label{P2.0}
	Denote  by ${U_r^{\br}:=\{U_r^{\br}(\tmt):t\geq 0\}}$   the L\'evy process with  Parisian reflection above  $b_{2}$ pushing the surplus process  to the  the level $b_{1}$ (with $0\leq b_{1}<b_{2}$),  i.e.,  the process is observed only at times  {belonging to the set} $\mathcal{T}_r$ and is  {pushed} down to the level $b_{1}$ if and only if it is  {observed} above $b_{2}$; for more details about its construction see Subsection \ref{Par1}. Let us define the first down-
	crossing time for the process $U^{\br}$ 
	 by $\tau_a^- (r;\br)\eqdef \inf \left\{ \tmt > 0: U^{\br}(\tmt)< a \right\}$,  $a \in \R$. 
	Then, it can be proven that 
	\begin{align}\label{eq16}
		&\E_{x}\Big[\expo^{-q\tau^{-}_{0}(r;\br)}; \tau^{-}_{0}(r;\br)<\infty\Big]\notag\\
		&=J^{(q,r)}_{b_{2}}(x)+r[Z^{(q)}(b_{2})-Z^{(q)}(b_{1})]\overline{W}^{(q+r)}(x-b_{2})\notag\\
		&\quad-\overline{C}^{(q,r)}_{\br}\{W^{(q)}(b_{2})I^{(q,r)}_{b_{2}}(x)+r\overline{W}^{(q+r)}(x-b_{2})[W^{(q)}(b_{2})-W^{(q)}(b_{1})]\},
	\end{align}
	where
	\begin{align}
		I^{(q,r)}_{b_{2}}(x)&\eqdef\frac{W^{(q,r)}_{b_{2}}(x)}{W^{(q)}(b_{2})}-r\overline{W}^{(q+r)}(x-b_{2}),\label{eq16.4}\\
		J^{(q,r)}_{b_{2}}(x)&\eqdef	Z^{(q,r)}_{b_{2}}(x)-rZ^{(q)}(b_{2})\overline{W}^{(q+r)}(x-b_{2}),\label{eq16.4.0}\\
		\overline{C}^{(q,r)}_{\br}
		&\eqdef\frac{r[Z^{(q)}(b_{2})-Z^{(q)}(b_{1})]+qZ^{(q)}(b_{2},\Phi(q+r))}{Z^{(q)\prime}(b_{2},\Phi(q+r))+r[W^{(q)}(b_{2})-W^{(q)}(b_{1})]}\notag\\
		&=\frac{g_{1}(b_{2}) Z^{(q)}(b_{2},\Phi(q+r))-rqW^{(q)}(b_{1})\xi(b_{1})}{\beta\{Z^{(q)\prime}(b_{2},\Phi(q+r))+r[W^{(q)}(b_{2})-W^{(q)}(b_{1})]\}}.\label{eq16.6}
	\end{align}
	For more details about its proof see Lemma \ref{L.A.3} and Appendix \ref{D}.
\end{rem}	

\begin{lema}\label{lem_2}
		Assume that $\xi(0+)\leq\tilde{g}(0+)$, i.e. $b^{*}_{1}=0$. For $x\in(b^{*}_{2},\infty)$, consider  $\bar{h}_{3}(x)$ and $\bar{h}_{4}(x)$ as follows
		\begin{align}\label{vp3.0.1}
			\begin{split}
				\bar{h}_{3}(x)&\eqdef a_{1}\bigg[1-r \int_{0}^{b^{*}_{2} } \frac{W^{(q + r)} (x-y)}{W^{(q+r)}(x)}W^{(q)} (y) \der y\bigg],\\
				\bar{h}_{4}(x)&\eqdef ra_{2} \frac{\overline{W}^{(q+r)}(x-b^{*}_{2})}{W^{(q+r)}(x)},
			\end{split}
		\end{align}
		with 
		\begin{align}\label{e2.6}
			a_{1}\eqdef\beta\overline{C}^{(q,r)}_{\br^{*}}-\frac{g_{1}(b^{*}_{2})}{\Phi(q+r)}\quad\text{
				and} \quad a_{2}\eqdef
			W^{(q)}(b^{*}_{1})\big[\beta  \overline{C}^{(q,r)}_{\br^{*}}-q\xi(b^{*}_{1})\big]. 
		\end{align}
		Then $\bar{h}_{3}$, $\bar{h}_{4}$ are  decreasing  and increasing on $(b^{*}_{2},\infty)$, respectively, which satisfy
		\begin{align}\label{vp3.2.0}
			\lim_{x\rightarrow\infty}\bar{h}_{3}(x)=\lim_{x\rightarrow\infty}\bar{h}_{4}(x).
		\end{align}
\end{lema}

\begin{proof}[Proof of Lemma \ref{P2}] 
	Notice that
	\begin{align}\label{vp1}
		v_{\alpha}^{\br^{*}\prime}(x)	&=-\frac{g_{1}(b^{*}_{2})}{\Phi(q+r)}{W}^{(q,r)}_{b^{*}_{2}}(x)-r{\overline{W}}^{(q+r)}(x-b^{*}_{2})+\beta{Z}^{(q,r)}_{b^{*}_{2}}(x),
	\end{align}	
	due to  \eqref{v6},  \eqref{v4}, \eqref{eq19} and \eqref{v10}.  On the other hand, by the proof of Lemma \ref{beh1}(i), we know that 	 
	\begin{equation}\label{cond1}
	\frac{g_{1}(b^{*}_{2})}{q\Phi(q+r)}=\xi(b^{*}_{1})\quad \text{if}\ \xi(0+)>\tilde{g}(0+). 
	\end{equation}
	This is true for both cases covered in Lemma \ref{beh2}(i), i.e. $b^{*}_{1}>0$. So, considering \eqref{cond1} and from   \eqref{eq5.0} and \eqref{eq16.6}, it can be verified that 
	\begin{equation}\label{vp1.1}
		\overline{C}^{(q,r)}_{\br^{*}}W^{(q)}(b^{*}_{1})-Z^{(q)}(b^{*}_{1})=-\frac{1}{\beta}\quad\text{and}\quad \beta \overline{C}^{(q,r)}_{\br^{*}}=\frac{g_{1}(b^{*}_{2})}{\Phi(q+r)},
	\end{equation}
	where $\overline{C}^{(q,r)}_{\br^{*}}$ is as in \eqref{eq16.6}.  Applying \eqref{vp1.1} in \eqref{vp1}, and using \eqref{eq16}, it follows that
		\begin{align*}
		v_{\alpha}^{\br^{*}\prime}(x)	&=\beta\Big\{{Z}^{(q,r)}_{b^{*}_{2}}(x)-\overline{C}^{(q,r)}_{\br^{*}}{W}^{(q,r)}_{b^{*}_{2}}(x)-r\Big[Z^{(q)}(b^{*}_{1})-\overline{C}^{(q,r)}_{\br^{*}}W^{(q)}(b^{*}_{1})\Big]{\overline{W}}^{(q+r)}(x-b^{*}_{2})\Big\}\notag\\
		&=\beta\E_{x}\Big[\expo^{-q\tau^{-}_{0}(r;\br^{*})}\Big].
		\end{align*}
	From here we conclude that $v_{\alpha}^{\br^{*}\prime}(x)\in(0,\beta)$ for $x>0$.  In case that $\xi(0+)\leq\tilde{g}(0+)$, i.e. $b^{*}_{1}=0$, let us check that  
	\begin{equation}\label{e2.3}
		v^{\br^{*}\prime}_{\alpha}(x)>\beta\E_{x}\Big[\expo^{-q\tau^{-}_{0}(r;\br^{*})}\Big]>0\quad \text{for}\  x\in(0,\infty). 
	\end{equation}
	Using \eqref{eq5.0} and  \eqref{eq16.6}, it gives that 
	\begin{align}\label{e2.2}
		\beta \overline{C}^{(q,r)}_{\br^{*}}&=\frac{g_{1}(b^{*}_{2}) Z^{(q)}(b^{*}_{2},\Phi(q+r))-rqW^{(q)}(b^{*}_{1})\xi(b^{*}_{1})}{Z^{(q)\prime}(b^{*}_{2},\Phi(q+r))+r[W^{(q)}(b^{*}_{2})-W^{(q)}(b^{*}_{1})]}\notag\\
		&>\frac{g_{1}(b^{*}_{2}) Z^{(q)}(b^{*}_{2},\Phi(q+r))-rW^{(q)}(b^{*}_{1})\frac{g_{1}(b^{*}_{2})}{\Phi(q+r)}}{Z^{(q)\prime}(b^{*}_{2},\Phi(q+r))+r[W^{(q)}(b^{*}_{2})-W^{(q)}(b^{*}_{1})]}=\frac{g_{1}(b^{*}_{2})}{\Phi(q+r)},
	\end{align}
	due to  $\xi(b^{*}_{1})<\frac{g_{1}(b^{*}_{2})}{q\Phi(q+r)}$, with $b^{*}_{1}=0$ and $b^{*}_{2}=b_{2}(0)$. By \eqref{eq16}, \eqref{vp1} and \eqref{e2.2}, it follows that \eqref{e2.3} is true for $x\in(0,b^{*}_{2})$. Using  \eqref{eq1} y \eqref{identities}, we obtain that
	\begin{align*}
		\begin{split}
			W_{b^{*}_{2}}^{(q, r)} (x ) 
			&= W^{(q+r)}(x) - r  \int_{0}^{b^{*}_{2} } W^{(q + r)} (x-y)W^{(q)} (y) \der y,\\
			Z_{b^{*}_{2}}^{(q, r)} (x ) 
			&=Z^{(q+r)}(x) - r  \int_{0}^{b^{*}_{2} } W^{(q + r)} (x-y)Z^{(q)} (y) \der y,
		\end{split}
		\quad\text{for}\ x\geq b^{*}_{2}.
	\end{align*}
	Then, applying those identities on \eqref{vp1}, we obtain that 
	\begin{align}\label{vp2.3.0}	
		v_{\alpha}^{\br^{*}\prime}(x)	
		&=\beta Z^{(q+r)}(x) -\frac{g_{1}(b^{*}_{2})}{\Phi(q+r)}W^{(q+r)}(x) -r \int_{0}^{x}W^{(q+r)}(y)\der y+r \int_{0}^{b^{*}_{2}}W^{(q+r)}(x-y)\der y\notag\\
		&\quad- r\beta  \int_{0}^{b^{*}_{2} } W^{(q + r)} (x-y)Z^{(q)} (y) \der y+r \frac{g_{1}(b^{*}_{2})}{\Phi(q+r)} \int_{0}^{b^{*}_{2} } W^{(q + r)} (x-y)W^{(q)} (y) \der y.
	\end{align}
	Then, taking $x\in[b^{*}_{2},\infty)$ and using  \eqref{eq16} and \eqref{vp2.3.0}, we get that 
	\begin{align*}
		v_{\alpha}^{\br^{*}\prime}(x)-\beta\E_{x}\Big[\expo^{-q\tau^{-}_{0}(r;\br^{*})}\Big]
		&=a_{1}{W}^{(q,r)}_{b^{*}_{2}}(x)-ra_{2}{\overline{W}}^{(q+r)}(x-b^{*}_{2})\\
		&=W^{(q+r)}(x) \big[\bar{h}_{3}(x)-\bar{h}_{4}(x)\big],
	\end{align*}
	where   $\bar{h}_{3}$ and $\bar{h}_{4}$ are as in \eqref{vp3.0.1}. From here and Lemma \ref{lem_2} we deduct easily that \eqref{e2.3} is true.
	

	Let us now check that $v_{\alpha}^{\br^{*}\prime}$ is strictly decreasing on $(0,a^{*}\wedge b^{*}_{2})$. If $x\in(0,b^{*}_{2})$,  by \eqref{vp1} and \eqref{vp1.1}, $v_{\alpha}^{\br^{*}\prime}(x)	=\beta{Z}^{(q)}(x)-\frac{g_{1}(b^{*}_{2})}{\Phi(q+r)}{W}^{(q)}(x).
	$
Then, calculating the first derivative of $v^{\br^{*}\prime}_{\alpha}$ on $(0,b^{*}_{2})$, we have that 
 \begin{align}\label{vp2}
 	v_{\alpha}^{\br^{*}\prime\prime}(x)	&=\beta q {W}^{(q)}(x)-\frac{g_{1}(b^{*}_{2})}{\Phi(q+r)}{W}^{(q)\prime}(x)=qW^{(q)\prime}(x)\bigg[\beta\frac{W^{(q)}(x)}{W^{(q)\prime}(x)}-\frac{g_{1}(b^{*}_{2})}{q\Phi(q+r)}\bigg].
 \end{align}
Since \eqref{i2} holds, we get that $-\frac{g_{1}(b^{*}_{2})}{q\Phi(q+r)}<-\xi(b^{*}_{2})$.  Then, applying this inequality in \eqref{vp2},  taking into account \eqref{a1}--\eqref{derzeta} and that ${W^{(q)}}/{W^{(q)\prime}}$ is increasing on $(0,\infty)$, it follows that 
\begin{equation*}
v_{\alpha}^{\br^{*}\prime\prime}(x)<qW^{(q)\prime}(x)\bigg[\beta\frac{W^{(q)}(x)}{W^{(q)\prime}(x)}-\xi(b^{*}_{2})\bigg]<qW^{(q)\prime}(x)[\xi(a^{*})-\xi(b^{*}_{2})]<0\ \text{for}\   x\in(0,a^{*}\wedge b^{*}_{2}), 
\end{equation*}
which implies that $v_{\alpha}^{\br^{*}\prime} $ is strictly decreasing on $(0, a^{*}\wedge b^{*}_{2})$. In the case that $b^{*}_{1}>0$,  since $v_{\alpha}^{\br^{*}\prime}(b^{*}_{1})=1$ and $b^{*}_{1}<b^{*}_{r}< a^{*}\wedge b^{*}_{2}$, we have that $v_{\alpha}^{\br^{*}\prime}>1$ on $(0,b^{*}_{1})$ and $v_{\alpha}^{\br^{*}\prime}<1$ on $(b^{*}_{1},  a^{*}\wedge b^{*}_{2})$. In case that $b_{1}^{*}=0$, it follows that $v^{\br^{*}\prime}_{\alpha}<1$ on $(0,a^{*}\wedge b^{*}_{2})$,  because of $v^{\br^{*}\prime}_{\alpha}(0+)=v^{\br_{0+}\prime}_{\alpha}(0+)<1$. 

In order to check  that $v_{\alpha}^{\br^{*}\prime}<1$ on $(a^{*}\wedge b^{*}_{2},\infty)$, observe that if $b^{*}_{2}\leq a^{*}$, we only  need to check that  \eqref{vp2.1} is true. However, if $a^{*}<b^{*}_{2}$, we have  $\xi(x)<\frac{g_{1}(x)}{q\Phi(q+r)}<\frac{g_{1}(b^{*}_{2})}{q\Phi(q+r)}$ for $x\in(a^{*},b_{2}^{*})$, due to Lemma \ref{g.1.2}(i)  and \eqref{i2}.  From here, 
\begin{equation*}
v_{\alpha}^{\br^{*}\prime}(x)=qW^{q}(x)\bigg[\xi(x)-\frac{g_{1}(b^{*}_{2})}{q\Phi(q+r)}\bigg]+1<1\quad \text{for}\  x\in(a^{*},b^{*}_{2}). 
\end{equation*}
To prove that 
\begin{equation}\label{vp2.1}
	v_{\alpha}^{\br^{*}\prime}(x)<1\quad\text{for}\ x>b^{*}_{2},
\end{equation}
it is enough to check that 
\begin{equation}\label{vp2.2}
	v^{\br^{*}\prime}_{\alpha}\leq v^{b_{r}^{*}\prime}_{0}\quad \text{on}\  (b^{*}_{2},\infty), 
\end{equation}
because of $v^{b_{r}^{*}\prime}_{0}(x)\in (0,1]$ for $x\in(b^{*}_{r},\infty)$; see Lemma 5.2 of \cite{NobPerYamYan}. From \eqref{vp2.3.0},  noting that
\begin{align*}
	v^{b^{*}_{r}\prime}_{0}(x)&=\beta Z^{(q+r)}(x) -\frac{g_{1}(b^{*}_{r})}{\Phi(q+r)}W^{(q+r)}(x) -r \int_{0}^{x}W^{(q+r)}(y)\der y+r \int_{0}^{b^{*}_{r}}W^{(q+r)}(x-y)\der y\notag\\
	&\quad- r\beta  \int_{0}^{b^{*}_{r} } W^{(q + r)} (x-y)Z^{(q)} (y) \der y+r \frac{g_{1}(b^{*}_{r})}{\Phi(q+r)} \int_{0}^{b^{*}_{r} } W^{(q + r)} (x-y)W^{(q)} (y) \der y,
\end{align*}
and using  \eqref{vp3}, we have that  $v^{\br^{*}\prime}_{\alpha}(x)-v^{b^{*}_{r}\prime}_{0}(x)=qW^{(q+r)}(x)[\bar{h}_{1}(x;b^{*}_{r},b^{*}_{2})-\bar{h}_{2}(x;b^{*}_{r},b^{*}_{2})]<0$ for $x\in(b^{*}_{2},\infty)$, due to Lemma \ref{lem_1}.  Therefore \eqref{vp2.2} is true.
 \end{proof}

\appendix

\section{Fluctuation identities}\label{A}
Define the first down- and up-crossing times of the process $X$, respectively, by
\begin{align*}
	\tau_a^- := \inf \left\{ \tmt > 0: X(\tmt) < a \right\} \quad \textrm{and} \quad \tau_a^+ := \inf \left\{ \tmt > 0: X(\tmt)>  a \right\}, \quad a \in \R.
\end{align*}
Then, by \cite[Theorem 8.1]{Kyp} and  \cite[Equation (5)]{AIZ2016}, it is known that for any $a > b$ and $x \leq a$,
\begin{align}
	\begin{split}
		\E_x \left[\expo^{-q \tau_a^+};\tau_a^+ < \tau_b^- \right] &= \frac {W^{(q)}(x-b)}  {W^{(q)}(a-b)}, \\
		\E_{x}\Big[\expo^{-q\tau^{-}_{b}-\theta [b-X(\tau^{-}_{b})]};\tau^{-}_{b}<\tau^{+}_{a}\Big]&= Z^{(q)}(x-b,\theta)-Z^{(q)}(a-b,\theta)\frac{W^{(q)}(x-b)}{W^{(q)}(a-b)}.
	\end{split}
	\label{laplace_in_terms_of_z}
\end{align}

Using Theorem 2.7(i)  of \cite{KKR}, We get  that for any $x\in[0,b]$ and  any  measurable function $h:\R\longrightarrow\R$,
\begin{align}\label{res.1}
	\E_x\left[ \int_0^{\tau_b^+\wedge\tau_0^-}\expo^{-q\tmt}h(X(\tmt))\der\tmt\right]&= \int_{0}^{b}h(y)\left\{\frac{W^{(q)}(x)W^{(q)}(b-y)}{W^{(q)}(b)}-W^{(q)}(x-y)\right\}\der y,
\end{align}

Considering $T(1)$ the first arrival of the Poisson process $N^{r}$ with parameter $r$, by \cite[Equation (49) and Lemma 4]{PY2018}, we have that if $a>b$ and $x<a$,
\begin{align}
	\E_{x}\Big[\expo^{-qT(1)};T(1)<\tau^{-}_{b}\wedge\tau^{+}_{a}\Big]&= r\bigg[\frac{\overline{W}^{(q+r)}(a-b)}{W^{(q+r)}(a-b)}W^{(q+r)}(x-b)-\overline{W}^{(q+r)}(x-b)\bigg],\label{N.1}\\
	\E_{x}\Big[\expo^{-qT(1)}{X}(T(1));T(1)<{\tau}^{-}_{b}\wedge{\tau}^{+}_{a}\Big]&=r\Bigg[\frac{\overline{\overline{W}}^{(q+r)}(a-b)}{W^{(q+r)}(a-b)}W^{(q+r)}(x-b)-\overline{\overline{W}}^{(q+r)}(x-b)\Bigg]\notag\\
	&\quad+br\bigg[\frac{\overline{W}^{(q+r)}(a-b)}{W^{(q+r)}(a-b)}W^{(q+r)}(x-b)-\overline{W}^{(q+r)}(x-b)\bigg].\label{N.2}
\end{align}	
Letting $a\uparrow\infty$ in \eqref{N.1}--\eqref{N.2}, and since $ \lim_{b\rightarrow\infty}\frac{\overline{{W}}^{(q+r)}(b)}{W^{(q+r)}(b)}= \lim_{b\rightarrow\infty}\frac{W^{(q+r)}(b)}{W^{(q+r)\prime}(b)}=\frac{1}{\Phi(q+r)}$, it gives that 
\begin{align}
	\E_{x}\Big[\expo^{-qT(1)};T(1)<\tau^{-}_{b}\Big]&=\frac{r}{r+q}\left[1-Z^{(q+r)}(x-b)+\frac{q+r}{\Phi(q+r)}W^{(q+r)}(x-b)\right],\label{N.3}\\
	\E_{x}\Big[\expo^{-qT(1)}{X}(T(1));T(1)<{\tau}^{-}_{b}\Big]&=r\Bigg[\frac{W^{(q+r)}(x-b)}{[\Phi(q+r)]^{2}}-\overline{\overline{W}}^{(q+r)}(x-b)\Bigg]\notag\\
	&\quad+\frac{rb}{r+q}\left[1-Z^{(q+r)}(x-b)+\frac{q+r}{\Phi(q+r)}W^{(q+r)}(x-b)\right].\label{N.4}
\end{align}	
Let $S=\{S(\tmt):\tmt\geq0\}$ and $R^{0}=\{R^{0}(\tmt):\tmt\geq0\}$ be defined respectively as  $S(\tmt)\eqdef \sup_{0\leq\tms \leq \tmt}\{X(\tms)\vee0\}$ and $R^{0}(\tmt)\eqdef \sup_{0\leq \tms\leq \tmt}\{-X(\tms)\vee0\}$.  We denote by $\widehat{Y}:=S-X$ and $Y:=X+R^{0}$,  which are a strong Markov processes. Observe that the process $R^{0}$ pushes $X$ upward whenever it attempts to down-cross the level $0$; as a result the process $Y$ only takes values on $[0, \infty)$.
The reader is referred to \cite{Kyp, Pistorius2004} for a complete introduction to the theory of L\'evy processes and their reflected processes.  

Taking $\hat{\tau}_a\eqdef\inf\{t>0:\widehat{Y}\in(a,\infty)\}$, by Proposition 2 in \cite{Pistorius2004},  it is known that 
\begin{equation}\label{funH}
\E_{-x}\left[\expo^{-q\hat{\tau}_a}\right]=Z^{(q)}(a-x)-qW^{(q)}(a-x)\frac{W^{(q)}(a)}{W^{(q)\prime}(a)},\qquad\text{$x\in[0,a]$.}
\end{equation}
Notice that  $H_{1}$ as \eqref{defH}, is defined by means of \eqref{funH} when $x=0$.
	
Similarly, given $b>0$, let $\eta_b\eqdef\inf\{t>0:Y\in(b,\infty)\}$ be the first entrance time of $Y$ into $(b,\infty)$ we know from Proposition 2 in \cite{Pistorius2004} that 
\begin{align} \label{upcrossing_time_reflected}
	\E_x\Big[\expo^{-q\eta_b}\Bigr]=\frac{Z^{(q)}(x)}{Z^{(q)}(b)}, \quad x \leq b.
\end{align}
In addition, we know from \cite[page 167]{AvrPalPis} that
\begin{align}
	\E_x\biggr[ \int_{[0,\eta_b]}\expo^{-q\tmt} \der R^{0}(\tmt)\biggl]&=- k^{(q)}(x) +k(b)\frac{Z^{(q)}(x)} {Z^{(q)}(b)}, \quad x \leq b, \label{capital_injection_identity_SN}
\end{align}
where  $k^{(q)}$ is as in \eqref{zeta}.

Additionally, by Theorem 2.8(iii) of \cite{KKR} we have that for any bounded measurable function $h:\R\longrightarrow\R$ with compact support, and $x\in[0,b]$,
\begin{align}\label{rrb}
\E_x\left[ \int_0^{\eta_b}\expo^{-q\tmt}h(Y(\tmt)) \der  t \right]=\frac{Z^{(q)}(x)}{Z^{(q)}(b)} \int_0^{b}W^{(q)}(b-y)h(y)dy- \int_0^{b}W^{(q)}(x-y)h(y)dy.
\end{align}
Then, by letting $b\to\infty$ in \eqref{rrb} we obtain
\begin{align}\label{rrbi}
\E_x\left[ \int_0^{\infty}\expo^{-q\tmt}h(Y(\tmt)) \der  t \right]=Z^{(q)}(x)\frac{\Phi(q)}{q} \int_0^{\infty}\expo^{-\Phi(q)y}h(y)dy- \int_0^{x}W^{(q)}(x-y)h(y)dy.
\end{align}

On the other hand, by an application of Lemma 2.1 in \cite{lrz} we obtain the following identities for $b<x < c$
\begin{align}\label{id.0}
\begin{split}	
	\E_x\left[\expo^{-(q+r)\tau_b^-} W^{(q)}(X(\tau_b^-));\tau_b^-<\tau^{+}_{c}\right]
	&=W^{(q,r)}_{b}(x)-W^{(q+r)}(x-b)\frac{W^{(q,r)}_{b}(c)}{W^{(q+r)}(c-b)},\\
	\E_x\left[\expo^{-(q+r)\tau_b^-} Z^{(q)}(X(\tau_b^-));\tau_b^-<\tau^{+}_{c}\right]  &=Z^{(q,r)}_{b}(x)-W^{(q+r)}(x-b)\frac{Z^{(q,r)}_{b}(c)}{W^{(q+r)}(c-b)}, \\
	\E_x\left[\expo^{-(q+r)\tau_b^-} Z^{(q)}(X(\tau_b^-),\theta);\tau_b^-<\tau^{+}_{c}\right] 
	&=Z^{(q,r)}_{b}(x,\theta)-W^{(q+r)}(x-b)\frac{Z^{(q,r)}_{b}(c,\theta)}{W^{(q+r)}(c-b)},
\end{split}
\end{align}
By  \eqref{W1}--\eqref{z_t} and  proceeding as the proof of Lemma A1 in \cite{PY2018},
\begin{align}\label{id.0.2}
\begin{split}	
	\lim_{c\uparrow\infty}\frac{W^{(q,r)}_{b}(c)}{W^{(q+r)}(c-b)}&=Z^{(q)}(b,\Phi(q+r)),\\
	\lim_{c\uparrow\infty}\frac{Z^{(q,r)}_{b}(c)}{W^{(q+r)}(c-b)}&=r \int_{b}^{\infty}\expo^{-\Phi(q+r)(y-b)}Z^{(q)}(y)\der y,\\
	\lim_{c\uparrow\infty}\frac{Z^{(q,r)}_{b}(c,\theta)}{W^{(q+r)}(c-b)}&=r \int_{b}^{\infty}\expo^{-\Phi(q+r)(y-b)}Z^{(q)}(y,\theta)\der y.
\end{split}
\end{align}
Then, letting $c\uparrow\infty$ in \eqref{id.0}, it follows that
\begin{align}\label{id.0.1}
\begin{split}	
	&\E_x\left[\expo^{-(q+r)\tau_b^-} W^{(q)}(X(\tau_b^-));\tau_b^-<\infty\right]\\
	&\qquad\qquad=W^{(q,r)}_{b}(x)-rW^{(q+r)}(x-b) \int_{b}^{\infty}\expo^{-\Phi(q+r)(y-b)}W^{(q)}(y)\der y,\\
	&\E_x\left[\expo^{-(q+r)\tau_b^-} Z^{(q)}(X(\tau_b^-));\tau_b^-<\infty\right] \\
	 &\qquad\qquad=Z^{(q,r)}_{b}(x)-rW^{(q+r)}(x-b) \int_{b}^{\infty}\expo^{-\Phi(q+r)(y-b)}Z^{(q)}(y)\der y, \\
	&\E_x\left[\expo^{-(q+r)\tau_b^-} Z^{(q)}(X(\tau_b^-),\theta);\tau_b^-<\infty\right] \\
	&\qquad\qquad=Z^{(q,r)}_{b}(x,\theta)rW^{(q+r)}(x-b) \int_{b}^{\infty}\expo^{-\Phi(q+r)(y-b)}Z^{(q)}(y,\theta)\der y.
\end{split}
\end{align}
\subsection{L\'evy processes with  Parisian reflection}\label{Par1}

We construct the L\'evy process with  Parisian reflection at $b_{2}$ introduced in Remark \ref{P2.0}, as follows: the process is observed only at times  {belonging to the set} $\mathcal{T}_r$ and is  {pushed} down to the level $b_{1}$ if and only if it is  {observed} above $b_{2}$. Formally,
\[
U_r^{\br}(\tmt) = X(\tmt), \quad \tmt \in [0, T_0^+(1)),
\]
where 
\begin{equation}\label{def_T_0_1}
	T_{\br}^+(1) := \inf \lbrace T(i) \in \mathcal{T}_r : X(T(i)) > b_{2} \rbrace.
\end{equation}
The process then jumps downward by $X(T_{\br}^+(1))-b_{1}$ so that $U_r^{\br}(T_{\br}^+(1)) = b_{1}$. For $T_{\br}^+(1) \leq \tmt < T_{\br}^+(2)  := \inf\{T(i) > T_{\br}^+(1):\; U^{\br}_r(T(i)-) > b_{2}\}$, we have $U_r^{\br}(\tmt) = X(\tmt) - (X(T_{\br}^+(1))-b_{1})$.  The process $U_r^{\br}$ can be constructed by repeating the procedure seen before.

Suppose $L_r^{\br}(\tmt)$ is the cumulative amount of  reflection until time $t \geq 0$. Then we have
\begin{align*}
	U_r^{\br}(\tmt) = X(\tmt) - L_r^{\br}(\tmt), \quad \tmt \geq 0,
\end{align*}
with
\begin{align*}
	L_r^{\br}(\tmt) := \sum_{T_{\br}^+(i) \leq \tmt} \left(U_r^{\br}(T_{\br}^+(i)-)-b_{1}\right), \quad \tmt \geq 0, 
\end{align*}
where $(T_{\br}^+(n); n \geq 1)$ can be constructed inductively by \eqref{def_T_0_1} in the following way
\begin{eqnarray*}T_{\br}^+(n+1) := \inf\{T(i) > T_{\br}^+(n):\; U_r^{\br}(T(i)-) >b_{2}\}, \quad n \geq 1.
\end{eqnarray*}

Let us take the first down- and up-crossing times for the process $U^{\br}$, respectively, by
\begin{align*}
	\tau_a^- (r;\br):= \inf \{ \tmt > 0: U^{\br}(\tmt)< a \} \quad \textrm{and} \quad \tau_a^+ (r;\br):= \inf \{ \tmt > 0: U^{\br}(\tmt) >  a \}, \quad a \in \R.
\end{align*}
The proofs of the results of this section can be  found  in Appendix \ref{D}.
\begin{prop}\label{L.A.}
Let $q>0$, $\theta\geq0$ and $\br=(b_{1},b_{2})$ such that  $0\leq b_{1}<b_{2}$. If $c>b_{2}$, then
\begin{align}
	g(x;\br,c,\theta)&\eqdef\E_{x}\Big[\expo^{-(q\tau^{-}_{0}(r;\br)-\theta U^{\br}_{r}(\tau^{-}_{0}(r;\br)))}; \tau^{-}_{0}(r;\br)<\tau^{+}_{c}(r;\br)\Big]\notag\\
	&=Z^{(q,r)}_{b_{2}}(x,\theta)-r\overline{W}^{(q+r)}(x-b_{2})Z^{(q)}(b_{1},\theta)\notag\\
	&\quad-C^{(q,r)}_{\br,c}(\theta)[W^{(q,r)}_{b_{2}}(x)-r\overline{W}^{(q+r)}(x-b_{2})W^{(q)}(b_{1})],\label{gb.1}
\end{align}
where 
\begin{equation}
	C^{(q,r)}_{\br,c}(\theta)\eqdef\frac{Z^{(q,r)}_{b_{2}}(c,\theta)-r\overline{W}^{(q+r)}(c-b_{2})Z^{(q)}(b_{1},\theta)}{W_{b_{2}}^{(q,r)}(c)-r\overline{W}^{(q+r)}(c-b_{2})W^{(q)}(b_{1})}.\label{gb.2}
\end{equation}	
\end{prop}

Letting $c\uparrow\infty$ in \eqref{gb.1}--\eqref{gb.2}, it follows the next  lemma. 
\begin{lema}\label{L.A.2}
	Let $q>0$, $\theta\geq0$ and $\br=(b_{1},b_{2})$ such that $0\leq b_{1}<b_{2}$. Then, for $x\in\R$,
	\begin{align}\label{eq15}
		\E_{x}\Big[\expo^{-(q\tau^{-}_{0}(r;\br)-\theta U^{\br}_{r}(\tau^{-}_{0}(r;\br)))}; \tau^{-}_{0}(r;\br)<\infty\Big]&=Z^{(q,r)}_{b_{2}}(x,\theta)-r\overline{W}^{(q+r)}(x-b_{2})Z^{(q)}(b_{1},\theta)\notag\\
		&\quad-\overline{C}^{(q,r)}_{\br}(\theta)[W^{(q,r)}_{b_{2}}(x)-r\overline{W}^{(q+r)}(x-b_{2})W^{(q)}(b_{1})],
	\end{align}
with
\begin{align}\label{eq15.0}
	\overline{C}^{(q,r)}_{\br}(\theta)
	&\eqdef\frac{\frac{r}{\Phi(q+r)-\theta}Z^{(q)}(b_{2},\theta)+\Big[\frac{q-\psi(\theta)}{\Phi(q+r)-\theta}\Big]Z^{(q)}(b_{2},\Phi(q+r))-\frac{r}{\Phi(q+r)}Z^{(q)}(b_{1},\theta)}{Z^{(q)}(b_{2},\Phi(q+r))-\frac{r}{\Phi(q+r)}W^{(q)}(b_{1})}.
\end{align}
\end{lema}
To finalize this section let us state the following lemma, which is a consequence of Lemma \ref{L.A.2} and will help us in the proof of  Proposition  \ref{p1}; see Appendix \ref{C}. 
\begin{lema}\label{L.A.3}
		Let $q>0$ and $\br=(b_{1},b_{2})$ such that $0\leq b_{1}<b_{2}$. Then, for $x\in\R$, \eqref{eq16} is true, and
\begin{align}
		&\E_{x}\Big[\expo^{-q\tau^{-}_{0}(r;\br)} U^{\br}_{r}(\tau^{-}_{0}(r;\br)); \tau^{-}_{0}(r;\br)<\infty\Big]\notag\\
	&=K^{(q,r)}_{b_{2}}(x)+r\overline{W}^{(q+r)}(x-b_{2})[l^{(q)}(b_{2})-l^{(q)}(b_{1})]\notag\\
	&\quad-\widehat{C}^{(q,r)}_{\br}\{W^{(q)}(b_{2})I^{(q,r)}_{b_{2}}(x)+r\overline{W}^{(q+r)}(x-b_{2})[W^{(q)}(b_{2})-W^{(q)}(b_{1})]\},\label{eq16.2}
\end{align}
with $l^{(q)}$, $I^{(q,r)}_{b_{2}}$  as in \eqref{v2}, \eqref{eq16.4}, respectively, and 
\begin{align}
	K^{(q,r)}_{b_{2}}(x)&\eqdef l^{(q,r)}_{b_{2}}(x)-rl^{(q)}(b_{2})\overline{W}^{(q+r)}(x-b_{2}),\label{eq16.5}\\
	l^{(q,r)}_{b_{2}}(x)&\eqdef l^{(q)}(x) + r  \int_{b_{2}}^{x } W^{(q + r)} (x-y)l^{(q)} (y) \der y,\label{eq16.1}\\
	\widehat{C}^{(q,r)}_{\br}&\eqdef
	\frac{r[l^{(q)}(b_{2})-l^{(q)}(b_{1})]+\frac{r}{\Phi(q+r)}Z^{(q)}(b_{2})+\Big[\frac{q-\psi'(0+)\Phi(q+r)}{\Phi(q+r)}\Big]Z^{(q)}(b_{2},\Phi(q+r))}{Z^{(q)\prime}(b_{2},\Phi(q+r))+r[W^{(q)}(b_{2})-W^{(q)}(b_{1})]}.\label{eq16.0}
\end{align}

\end{lema}

\section{Proof of Proposition \ref{p1}}\label{C}
In order to verify the expression given in  \eqref{v1}, we only need to find the expressions for 
\begin{align*}
	f(x;\br)&\eqdef\E_{x}\bigg[ \int_{0}^{\infty}\expo^{-q  \tmt}\der\bigg[L^{0,\br}_{r}(\tmt)-\alpha\sum_{0\leq\tms\leq\tmt} \uno_{\{\Delta L^{\pi}_{r}(\tms)\neq0\}}\bigg]\bigg],\\
	\hat{f}(x;\br)&\eqdef\E_{x}\bigg[ \int_{0}^{\infty}\expo^{-q  \tmt}\der R^{0,\br}_{r}(\tmt)\bigg].\\
\end{align*}
\begin{lema}\label{L2}
If $x\in\R$ and $\br=(b_{1},b_{2})$ such that $0\leq b_{1}<b_{2}$, then
\begin{align}
	f(x;\br)&=\frac{r\Big[Z^{(q,r)}_{b_{2}}(x)-rZ^{(q)}(b_{1})\overline{W}^{(q+r)}(x-b_{2})\Big]\bigg[\frac{1}{\Phi(q+r)}+b_{2}-b_{1}-\alpha\bigg]}{qZ^{(q)}(b_{2},\Phi(q+r))+r[Z^{(q)}(b_{2})-Z^{(q)}(b_{1})]}\notag\\
	&\quad-r\overline{\overline{W}}^{(q+r)}(x-b_{2})-r\overline{W}^{(q+r)}(x-b_{2})[b_{2}-(b_{1}+\alpha)],\label{a.3}\\
	\hat{f}(x;\br)&=-K^{(q,r)}_{b_{2}}(x)-r\overline{W}^{(q+r)}(x-b_{2})[l^{(q)}(b_{2})-l^{(q)}(b_{1})]\notag\\
	&\quad+C^{(q,r)}_{\br}\Big\{J^{(q,r)}_{b_{2}}(x)+r[Z^{(q)}(b_{2})-Z^{(q)}(b_{1})]\overline{W}^{(q+r)}(x-b_{2})\Big\},\label{a.4}
\end{align}
with
\begin{align}
	C^{(q,r)}_{\br}&\eqdef\frac{r[l^{(q)}(b_{2})-l^{(q)}(b_{1})]}{r[Z^{(q)}(b_{2})-Z^{(q)}(b_{1})]+qZ^{(q)}(b_{2},\Phi(q+r))}\notag\\
	&\quad+\frac{rZ^{(q)}(b_{2})+[q-\psi'(0+)\Phi(q+r)]Z^{(q)}(b_{2},\Phi(q+r))}{\Phi(q+r)\{r[Z^{(q)}(b_{2})-Z^{(q)}(b_{1})]+qZ^{(q)}(b_{2},\Phi(q+r))\}}.\label{a.5}
\end{align}
\end{lema}

\begin{proof}[Proof of Proposition \ref{p1}]
Applying  \eqref{a.3}--\eqref{a.4} in \eqref{va1} an taking into account \eqref{eq16.4.0}, \eqref{eq16.5} and  \eqref{a.5}, it gives the expression in \eqref{v1}. 
\end{proof}
\begin{proof}[Proof of  Lemma \ref{L2}. Equation \eqref{a.3}]
	If $x<b_{2}$, by strong Markov property and \eqref{upcrossing_time_reflected},  we get that 
	\begin{align}\label{eq0}
		f(x;\br)
		&=\E_{x}[\expo^{-q \eta_{b_{2}}}]f(b_{2};\br)=\frac{Z^{(q)}(x)}{Z^{(q)}(b_{2})}f(b_{2};\br).
	\end{align}
	If $x>b_{2}$, notice that 
	\begin{align}\label{eq0.1}
		f(x;\br)&=\E_{x}\bigg[ \int_{0}^{\infty}\expo^{-q  \tmt}\der\bigg[L^{0,\br}_{r}(\tmt)-\alpha\sum_{0\leq\tms\leq\tmt} \uno_{\{\Delta L^{\pi}_{r}(\tms)\neq0\}}\bigg];T(1)<\tau^{-}_{b_{2}}\bigg]\notag\\
		&\quad+\E_{x}\bigg[ \int_{0}^{\infty}\expo^{-q  \tmt}\der\bigg[L^{0,\br}_{r}(\tmt)-\alpha\sum_{0\leq\tms\leq\tmt} \uno_{\{\Delta L^{\pi}_{r}(\tms)\neq0\}}\bigg];\tau^{-}_{b_{2}}<T(1)\bigg]\defeq f_{1}(x;\br)+f_{2}(x;\br).
	\end{align}
	By strong Markov property and considering \eqref{N.3}--\eqref{N.4} and \eqref{eq0}, we obtain that
	\begin{align}\label{f1}
		f_{1}(x;\br)
		&=\E_{x}\Big[\expo^{-qT(1)}(X(T(1))-b_{1}-\alpha);T(1)<\tau^{-}_{b_{2}}\Big]+\E_{x}\bigg[\expo^{-qT(1)};T(1)<\tau^{-}_{b_{2}}\bigg]f(b_{1};\br)\notag\\
		&=\E_{x}\Big[\expo^{-qT(1)}X(T(1));T(1)<\tau^{-}_{b_{2}}\Big]+[f(b_{1};\br)-(b_{1}+\alpha)]\E_{x}\bigg[\expo^{-qT(1)};T(1)<\tau^{-}_{b_{2}}\bigg]\notag\\
		&=r\bigg[\frac{W^{(q+r)}(x-b_{2})}{[\Phi(q+r)]^2}-\overline{\overline{W}}^{(q+r)}(x-b_{2})\bigg]\notag\\
		&\quad+\frac{r}{r+q}\bigg[\frac{Z^{(q)}(b_{1})}{Z^{(q)}(b_{2})}f(b_{2};\br)+b_{2}-(b_{1}+\alpha)\bigg]\left[1-Z^{(q+r)}(x-b_{2})+\frac{q+r}{\Phi(q+r)}W^{(q+r)}(x-b_{2})\right] .
	\end{align}
	On the other hand, by \eqref{id.0.1} and \eqref{eq0}, it implies
	\begin{align}\label{eq7}
		f_{2}(x;\br)
		&=\E_{x}\Big[\expo^{-q\tau^{-}_{b_{2}}}f(X(\tau^{-}_{b_{2}});\br);\tau^{-}_{b_{2}}<T(1)\Big]\notag\\
		&=\frac{f(b_{2},\br)}{Z^{(q)}(b_{2})}\E_{x}\Big[\expo^{-q\tau^{-}_{b_{2}}}Z^{(q)}(X(\tau^{-}_{b_{2}}));\tau^{-}_{b_{2}}<T(1)\Big]\notag\\
		&=\frac{f(b_{2},\br)}{Z^{(q)}(b_{2})}\E_{x}\Big[\expo^{-(q+r)\tau^{-}_{b_{2}}}Z^{(q)}(X(\tau^{-}_{b_{2}}));\tau^{-}_{b_{2}}<\infty\Big]\notag\\
		&=\frac{f(b_{2},\br)}{Z^{(q)}(b_{2})}\bigg[Z^{(q,r)}_{b_{2}}(x)-rW^{(q+r)}(x-b_{2}) \int_{b_{2}}^{\infty}\expo^{-\Phi(q+r)(y-b_{2})}Z^{(q)}(y)\der y\bigg].
	\end{align}
	Then, applying \eqref{f1}--\eqref{eq7} in \eqref{eq0.1}, we get 
	\begin{align}
		f(x,\br)
		&=\frac{f(b_{2};\br)}{Z^{(q)}(b_{2})}Z^{(q,r)}_{b_{2}}(x)-r\overline{\overline{W}}^{(q+r)}(x-b_{2})\notag\\
		&\quad+\frac{r}{r+q}[1-Z^{(q+r)}(x-b_{2})]\bigg[[b_{2}-(b_{1}+\alpha)]+\frac{f(b_{2};\br)}{Z^{(q)}(b_{2})}Z^{(q)}(b_{1})\bigg]\notag\\
		&\quad+r\frac{W^{(q+r)}(x-b_{2})}{\Phi(q+r)}\bigg\{\frac{1}{\Phi(q+r)}+b_{2}-b_{1}-\alpha\notag\\
		&\quad+\frac{f(b_{2};\br)}{Z^{(q)}(b_{2})}\Bigg[Z^{(q)}(b_{1})-\Phi(q+r) \int_{b_{2}}^{\infty}\expo^{-\Phi(q+r)(y-b_{2})}Z^{(q)}(y)\der y\bigg]\bigg\}.\label{f2}
	\end{align}
	Now we shall  estimate $f(b_{2};\br)$. Take the process $Y$ as in Appendix \ref{A}. Then, 
	\begin{align}
		f(b_{2};\br)&=\E_{b_{2}}\bigg[ \int_{0}^{\infty}\expo^{-q  \tmt}\der\bigg[L^{0,\br}_{r}(\tmt)-\alpha\sum_{0\leq\tms\leq\tmt} \uno_{\{\Delta L^{\pi}_{r}(\tms)\neq0\}}\bigg];Y(T(1))\geq b_{2}\bigg]\notag\\
		&\quad+\E_{b_{2}}\bigg[ \int_{0}^{\infty}\expo^{-q  \tmt}\der\bigg[L^{0,\br}_{r}(\tmt)-\alpha\sum_{0\leq\tms\leq\tmt} \uno_{\{\Delta L^{\pi}_{r}(\tms)\neq0\}}\bigg];Y(T(1))<b_{2}\bigg]\notag\\
		&\defeq\bar{f}_{1}(b_{2};\br)+\bar{f}_{2}(b_{2};\br).\label{f3}
	\end{align}
	By \eqref{rrbi}, \eqref{eq0} and strong Markov property, it is obtained  that
	\begin{align}
		\bar{f}_{1}(b_{2};\br)
		&=\E_{b_{2}}\Big[\expo^{-qT(1)}(Y(T(1))-b_{1}-\alpha);Y(T(1))\geq b_{2}\Big]\notag\\
		&\quad+\E_{b_{2}}\bigg[\expo^{-qT(1)}f(b_{1};\br);Y(T(1))\geq b_{2}\bigg]\notag\\
		&=\E_{b_{2}}\bigg[\expo^{-qT(1)}\bigg(Y(T(1))+\frac{Z^{(q)}(b_{1})}{Z^{(q)}(b_{2})}f(b_{2};\br)-b_{1}-\alpha\bigg);Y(T(1))\geq b_{2}\bigg]\notag\\
		&=r\E_{b_{2}}\bigg[ \int_{0}^{\infty}\expo^{-(q+r)\xi}\bigg(Y(\xi)+\frac{Z^{(q)}(b_{1})}{Z^{(q)}(b_{2})}f(b_{2};\br)-b_{1}-\alpha\bigg)\uno_{\{Y(\xi)\geq b_{2}\}}\der \xi\bigg]\notag\\
		&=rZ^{(q+r)}(b_{2})\frac{\Phi(q+r)}{q+r} \int_{b_{2}}^{\infty}\expo^{-\Phi(q+r)y}\bigg[y+\frac{Z^{(q)}(b_{1})}{Z^{(q)}(b_{2})}f(b_{2};\br)-b_{1}-\alpha\bigg]\der y\notag\\
		&=rZ^{(q+r)}(b_{2})\frac{\expo^{-\Phi(q+r)b_{2}}}{q+r}\bigg[\frac{(\Phi(q+r)b_{2}+1)}{\Phi(q+r)}+\frac{Z^{(q)}(b_{1})}{Z^{(q)}(b_{2})}f(b_{2};\br)-b_{1}-\alpha\bigg]\notag\\
		&=\frac{r}{q+r}\expo^{-\Phi(q+r)b_{2}}Z^{(q+r)}(b_{2})\frac{Z^{(q)}(b_{1})}{Z^{(q)}(b_{2})}f(b_{2};\br)\notag\\
		&\quad+rZ^{(q+r)}(b_{2})\frac{\expo^{-\Phi(q+r)b_{2}}}{q+r}\bigg[\frac{1}{\Phi(q+r)}+b_{2}-b_{1}-\alpha\bigg].
	\end{align}
	Meanwhile, using again \eqref{rrbi}, \eqref{eq0} and strong Markov property,
	\begin{align*}
		\bar{f}_{2}(b_{2};\br)&=\E_{b_{2}}\Big[\expo^{-qT(1)}{f}(Y(T(1));\br);Y(T(1))<b_{2}\Big]\notag\\
		&=r\E_{b_{2}}\bigg[ \int_{0}^{\infty}\expo^{-(q+r)\xi}\frac{Z^{(q)}(Y(\xi))}{Z^{(q)}(b_{2})}f(b_{2};\br)\uno_{\{Y(\xi)< b_{2}\}}\der \xi\bigg]\notag\\
		&=\frac{f(b_{2};\br)}{Z^{(q)}(b_{2})}\bigg[rZ^{(q+r)}(b_{2})\frac{\Phi(q+r)}{q+r} \int_{0}^{b_{2}}\expo^{-\Phi(q+r)y}Z^{(q)}(y)\der y\notag\\
		&\quad-r \int_{0}^{b_{2}}W^{(q+r)}(b_{2}-y)Z^{(q)}(y)\der y\bigg].
	\end{align*}
	Using integration by parts together with \eqref{scale_function_laplace}, it follows that
	\begin{align*}
		 \int_0^{\infty}\expo^{-\Phi(q+r)y}Z^{(q)}(y) \der y=\frac{1}{\Phi(q+r)}+\frac{q}{\Phi(q+r)} \int_0^{\infty}\expo^{-\Phi(q+r)y}W^{(q)}(y) \der y=\frac{r+q}{r\Phi(q+r)}.
	\end{align*}
	Hence,
	\begin{align*}
		r \int_0^{b_{2}}\expo^{-\Phi(q+r)y}Z^{(q)}(y) \der y=\frac{r+q}{\Phi(q+r)}-r \int_{b_{2}}^{\infty}\expo^{-\Phi(q+r)y}Z^{(q)}(y) \der y.
	\end{align*}
	From here and  by \eqref{eq1},
	it implies
	\begin{align}
		\bar{f}_{2}(b_{2};\br)
		&=f(b_{2};\br)\bigg[1-\frac{r\Phi(q+r)}{q+r}\frac{Z^{(q+r)}(b_{2})}{Z^{(q)}(b_{2})} \int_{b_{2}}^{\infty}\expo^{-\Phi(q+r)y}Z^{(q)}(y) \der y\bigg].\label{f4}
	\end{align}
	Thus, by \eqref{f3}--\eqref{f4},
	\begin{align*}
		f(b_{2};\br)
		&=rZ^{(q+r)}(b_{2})\frac{\expo^{-\Phi(q+r)b_{2}}}{q+r}\bigg[\frac{1}{\Phi(q+r)}+b_{2}-b_{1}-\alpha\bigg]\notag\\
		&\quad+f(b_{2};\br)\bigg\{\frac{r}{q+r}\expo^{-\Phi(q+r)b_{2}}Z^{(q+r)}(b_{2})\frac{Z^{(q)}(b_{1})}{Z^{(q)}(b_{2})}\notag\\
		&\quad+1-\frac{r\Phi(q+r)}{q+r}\frac{Z^{(q+r)}(b_{2})}{Z^{(q)}(b_{2})} \int_{b_{2}}^{\infty}\expo^{-\Phi(q+r)y}Z^{(q)}(y) \der y\bigg\}.
	\end{align*}
	From here it can be verified that 
	\begin{align}\label{f5}
		0&=\bigg[\frac{1}{\Phi(q+r)}+b_{2}-b_{1}-\alpha\bigg]\notag\\
		&\quad+f(b_{2};\br)\frac{1}{Z^{(q)}(b_{2})}\bigg\{Z^{(q)}(b_{1})-\Phi(q+r) \int_{b_{2}}^{\infty}\expo^{-\Phi(q+r)(y-b_{2})}Z^{(q)}(y) \der y\bigg\},
	\end{align}
	which  implies
	\begin{equation}\label{f6}
		f(b_{2};\br)=\frac{rZ^{(q)}(b_{2})\bigg[\frac{1}{\Phi(q+r)}+b_{2}-b_{1}-\alpha\bigg]}{qZ^{(q)}(b_{2},\Phi(q+r))+r[Z^{(q)}(b_{2})-Z^{(q)}(b_{1})]},
	\end{equation}
	since by integration by parts, we get that 
	\begin{align*}
		\Phi(q+r) \int_{b_{2}}^{\infty}\expo^{-\Phi(q+r)(y-b_{2})}Z^{(q)}(y)\der y-Z^{(q)}(b_{2})&=q \int_{b_{2}}^{\infty}\expo^{-\Phi(q+r)(y-b_{2})}W^{(q)}(y)\der y\notag\\
		&=\frac{q}{r}Z^{(q)}(b_{2},\Phi(q+r)).
	\end{align*}
	Therefore, applying \eqref{f5}--\eqref{f6} in \eqref{f2}, we conclude that \eqref{a.3} is true.
\end{proof}

\begin{proof} [Proof of  Lemma \ref{L2}. Equation \eqref{a.4}]

If $x<b_{2}$, by \eqref{upcrossing_time_reflected}--\eqref{capital_injection_identity_SN} and the strong Markov property, we get
\begin{align*}
	\hat{f}(x;\br)
	&=\E_{x}\bigg[ \int_{0}^{\eta_{b_{2}}}\expo^{-q  \tmt}\der R^{0,\br}_{r}(\tmt)\bigg]+\E_{x}\Big[\expo^{-q\eta_{b_{2}}}\Big]\hat{f}(b_{2};\br)\notag\\
	&=\frac{Z^{(q)}(x)}{Z^{(q)}(b_{2})}[k^{(q)}(b_{2})+\hat{f}(b_{2};\br)]-k^{(q)}(x),
\end{align*}
In particular,
\begin{equation}\label{f7}
	\hat{f}(0;\br)=\frac{1}{Z^{(q)}(b_{2})}[k^{(q)}(b_{2})+\hat{f}(b_{2};\br)]-\frac{\psi'(0+)}{q}.
\end{equation}	
If $x\geq b_{2}$, by \eqref{eq16}, \eqref{eq16.2} and \eqref{f7}, it gives
\begin{align}
	\hat{f}(x;\br)&=\E_{x}\bigg[ \int_{\tau^{-}_{0}(r;\br)}^{\infty}\expo^{-q  \tmt}\der R^{0,\br}_{r}(\tmt)\bigg]\notag\\
	&=\E_{x}\big[\expo^{-q\tau^{-}_{0}(r;\br)}[-U^{\br}_{r}(\tau^{-}_{0}(r;\br))+\hat{f}(0;\br)]\big]\notag\\
	&=-\Big\{K^{(q,r)}_{b_{2}}(x)+r\overline{W}^{(q+r)}(x-b_{2})[l^{(q)}(b_{2})-l^{(q)}(b_{1})]\notag\\
	&\quad-\widehat{C}^{(q,r)}_{\br}\{W^{(q)}(b_{2})I^{(q,r)}_{b_{2}}(x)+r\overline{W}^{(q+r)}(x-b_{2})[W^{(q)}(b_{2})-W^{(q)}(b_{1})]\}\Big\}\notag\\
	&\quad+\bigg[\frac{1}{Z^{(q)}(b_{2})}[l^{(q)}(b_{2})+\hat{f}(b_{2};\br)]\bigg]\{J^{(q,r)}_{b_{2}}(x)+r[Z^{(q)}(b_{2})-Z^{(q)}(b_{1})]\overline{W}^{(q+r)}(x-b_{2})\notag\\
	&\quad-\overline{C}^{(q,r)}_{\br}\{W^{(q)}(b_{2})I^{(q,r)}_{b_{2}}(x)+r\overline{W}^{(q+r)}(x-b_{2})[W^{(q)}(b_{2})-W^{(q)}(b_{1})]\}\},\label{f8}
\end{align}
with $\overline{C}^{(q,r)}_{\br}$, $\widehat{C}^{(q,r)}_{\br}$ as in \eqref{eq16.6}, \eqref{eq16.0}, respectively. Taking $x=b_{2}$ in \eqref{f8}, 
\begin{align*}
	\hat{f}(b_{2};\br)&=-\Big\{K^{(q,r)}_{b_{2}}(b_{2})-\widehat{C}^{(q,r)}_{\br}W^{(q)}(b_{2})I^{(q,r)}_{b_{2}}(b_{2})\Big\}\notag\\
	&\quad+\bigg[\frac{1}{Z^{(q)}(b_{2})}[l^{(q)}(b_{2})+\hat{f}(b_{2};\br)]\bigg]\{J^{(q,r)}_{b_{2}}(b_{2})-\overline{C}^{(q,r)}_{\br}W^{(q)}(b_{2})I^{(q,r)}_{b_{2}}(b_{2})\}.
\end{align*}
From here and noticing that  $I^{(q,r)}_{b_{2}}(b_{2})=1$, $J^{(q,r)}_{b_{2}}(b_{2})=Z^{(q)}(b_{2})$ and $K^{(q,r)}_{b_{2}}(b_{2})=l^{(q)}(b_{2})$; due to  \eqref{eq16.4}, \eqref{eq16.4.0} and \eqref{eq16.5}, it can be checked that 
\begin{align*}
	\frac{\hat{f}(b_{2};\br)}{Z^{(q)}(b_{2})}&=\frac{\widehat{C}^{(q,r)}_{\br}}{\overline{C}^{(q,r)}_{\br}}-\frac{l^{(q)}(b_{2})}{Z^{(q)}(b_{2})}.	
\end{align*}
Then,
\begin{align*}
	\hat{f}(x;\br)
	&=-K^{(q,r)}_{b_{2}}(x)-r\overline{W}^{(q+r)}(x-b_{2})[l^{(q)}(b_{2})-l^{(q)}(b_{1})]\notag\\
	&\quad+\frac{\widehat{C}^{(q,r)}_{\br}}{\overline{C}^{(q,r)}_{\br}}\Big\{J^{(q,r)}_{b_{2}}(x)+r[Z^{(q)}(b_{2})-Z^{(q)}(b_{1})]\overline{W}^{(q+r)}(x-b_{2})\Big\}.
\end{align*}
Taking into account \eqref{eq16.6} and \eqref{eq16.0}, it can be verified    $C^{(q,r)}_{\br}={\widehat{C}^{(q,r)}_{\br}}/{\overline{C}^{(q,r)}_{\br}}$ is as in \eqref{a.5}.
\end{proof}

\section{Proof of Lemmas \ref{Lh} and \ref{L_H}}\label{B}

\begin{proof}[Proof of Lemma \ref{Lh}]
	Let $h^{(q,r)}$ be as in \eqref{ineq3}. By \eqref{W1} and \eqref{def_z_nuevo}, we obtain 
	\begin{align}\label{b.1}
		h^{(q,r)}(u)
		&= \int_{0}^{\infty}\expo^{-(\Phi(q+r)-\Phi(q))z}\frac{W_{\Phi(q)}(z+u)}{W_{\Phi(q)}(u)}\der z.
	\end{align}
We note that
		\begin{equation}\label{bound_dc}
		\bigg|\frac{\der}{\der u}\bigg[\frac{W_{\Phi(q)}(z+u)}{W_{\Phi(q)}(u)}\bigg]\bigg|<\frac{W'_{\Phi(q)}(u)}{\psi'(\Phi(q))[W_{\Phi(q)}(u)]^{2}}.
		\end{equation}
		Hence, using the fact that the right-hand side of \eqref{bound_dc} is bounded on compact sets away from zero we can apply dominated convergence to obtain
	\begin{align*}
		h^{(q,r)\prime}(u)&= \int_{0}^{\infty}\expo^{-(\Phi(q+r)-\Phi(q))z}\frac{\der}{\der u}\bigg[\frac{W_{\Phi(q)}(z+u)}{W_{\Phi(q)}(u)}\bigg]\der z<0,
	\end{align*}
	It implies that  $h^{(q,r)}$ is decreasing on $(0,\infty)$. Letting $u\downarrow0$ in \eqref{ineq3} and using \eqref{scale_function_laplace}--\eqref{r1}, it easy to see that 
	\begin{align*}
		\lim_{u\downarrow0}h^{(q,r)}(u)
		&=\frac{  \int_{0}^{\infty}\expo^{-\Phi(q+r)z}W^{(q)}(z)\der z}{W^{(q)}(0)}=\frac{1}{rW^{(q)}(0)}.
	\end{align*}
	Now, letting $u\uparrow\infty$ in \eqref{b.1} and using  \eqref{W2}, we get that
	\begin{align*}
		\lim_{u\uparrow\infty}h^{(q,r)}(u)&=\frac{1}{\Phi(q+r)-\Phi(q)}.\qedhere
	\end{align*}	
\end{proof}

\begin{proof}[Proof of Lemma \ref{L_H}]
	Let us first verify that  for each $u\in(0,\infty)$ fixed, $r\mapsto	\Phi(q+r)-\frac{1}{h^{(q,r)}(u)}$ is increasing on $(0,\infty)$. For this, notice that it is sufficient to check that  for each $u$ fixed, $r\mapsto\frac{1}{h^{(q,r)}(u)}$ is decreasing on $(0,\infty)$ since $r\mapsto\Phi(q+r)$ is increasing on $(0,\infty)$. Considering $S(\tmt)\eqdef \sup_{0\leq\tms\leq\tmt}\{X(\tms)-x\}$, with $X(0)=x$, and $S^{-1}(\tmt)\eqdef\inf\{\tms:S(\tms)>\tmt\}$,  by Section 6.5.2 in \cite{Kyp}, it is well known that $\Phi(\eta)$ admits the following expression $\Phi(\eta)=-\log\big(\E\big[\expo^{-\eta S^{-1}(1)}\big]\big)$. From here, we see that  $\frac{\der}{\der r}[\Phi(q+r)]=\frac{\E\big[S^{-1}(1)\expo^{-[q+r] S^{-1}(1)}\big]}{\E\big[\expo^{-[q+r] S^{-1}(1)}\big]}\geq0$ for $r\geq0$. Then,  using \eqref{b.1},
	\begin{align*}
		\frac{\der}{\der r }\bigg[\frac{1}{h^{(q,r)}(u)}\bigg]=-\frac{\der}{\der r}[\Phi(q+r)]\frac{1}{[h^{(q,r)}(u)]^{2}} \int_{0}^{\infty}z\expo^{-(\Phi(q+r)-\Phi(q))z}\frac{W_{\Phi(q)}(z+u)}{W_{\Phi(q)}(u)}\der z\leq0,
	\end{align*}
	Taking $r_{1}$ and $r_{2}$ such that $0<r_{1}<r_{2}$. we have that  
	\begin{equation*}
		\Phi(q+r_{1})-\frac{1}{h^{(q,r_{1})}(u)}<\Phi(q+r_{2})-\frac{1}{h^{(q,r_{2})}(u)}\quad \text{for}\ u\in(0,\infty).
	\end{equation*} 
	From here, and by \eqref{eq5.0} and \eqref{ineq3}, it is easy to check that \eqref{ineq1} holds. To verify that \eqref{ineq2} is true, by the definition of $H_{1}$ and $H^{(q,r)}$, we see \eqref{ineq2} is equivalent to show that 
	\begin{equation*}
		\frac{Z^{(q)\prime}(u,\Phi(q+r))}{Z^{(q)}(u,\Phi(q+r))}<\frac{W^{(q)\prime}(u)}{W^{(q)}(u)}\quad\text{on}\ (0,\infty),
	\end{equation*}
	which holds, because of 
	\begin{equation*}
		\frac{Z^{(q)\prime}(u,\Phi(q+r))}{Z^{(q)}(u,\Phi(q+r))}-\frac{W^{(q)\prime}(u)}{W^{(q)}(u)}=\frac{h^{(q,r)\prime}(u)}{h^{(q,r)}(u)}<0\quad\text{for}\ u\in(0,\infty). 	\qedhere
	\end{equation*}	
\end{proof}

\section{Proof of Lemma \ref{g.1.2}}\label{D_nuevo}
\begin{proof}[Proof of Lemma \ref{g.1.2}] 
	(i) Taking first derivative in \eqref{v6}, it gives 
		\begin{align}
			g'_{1}(u)&=\frac{rq\beta W^{(q)}(u)Z^{(q)}(u,\Phi(q+r))-r[\beta Z^{(q)}(u)-1]Z^{(q)\prime}(u,\Phi(q+r))}{[Z^{(q)}(u,\Phi(q+r))]^{2}}\notag\\
			&=-\frac{rqW^{(q)}(u)}{Z^{(q)}(u,\Phi(q+r))}\bigg(\frac{ \beta H^{(q,r)}(u) -1 }{ Z^{(q)}(u) - H^{(q,r)}(u) }\bigg), \label{e4} 
		\end{align}
		with $H^{(q,r)}$ as in \eqref{fpt_pr}. From \eqref{e4} and \eqref{b2}, we conclude that $g_{1}$ attains its unique minimum at $b^{*}_{r}\geq0$.  Additionally, it is easy to verify that  $\lim_{u\downarrow0}g_{1}(u)=q\beta+ r(\beta-1)$. On the other hand,  by \eqref{eq5.0} and  \eqref{e1},
		\begin{align*}
			\lim_{u\uparrow\infty}g_{1}(u)&=q\beta+ \lim_{u\uparrow\infty}\frac{rq\beta W^{(q)}(u)}{Z^{(q)\prime}(u,\Phi(q+r))}\notag\\
			&=q\beta+ \lim_{u\uparrow\infty}\frac{q\beta }{\Phi(q+r)h^{(q,r)}(u)-1}\notag\\
			&=q\beta+ \frac{q\beta[\Phi(q+r)-\Phi(q)] }{\Phi(q)}=q\beta\frac{\Phi(q+r)}{\Phi(q)}.
	\end{align*}	
	
	(ii) First of all, by definition of $g_{2}(\cdot;b_{1})$, it is easy to check that $\lim_{u\downarrow b_{1}}g_{2}(u;b_{1})=\infty$. On the other hand, By  L'Hopital's rule and \eqref{L_xi}, we obtain that  $\lim_{u\uparrow\infty}g_{2}(u;b_{1})=\lim_{u\uparrow\infty}\xi(u)=\frac{\beta}{\Phi(q)}$. Now, for each $b_{1}\geq0$ fixed, let  us study  \eqref{v7}. Taking first derivatives w.r.t. $u$, we get that
	\begin{align}
		g'_{2}(u;b_{1})&=\frac{[\beta Z^{(q)}(u)-1]}{Z^{(q)}(u)-Z^{(q)}(b_{1})}-\frac{qW^{(q)}(u)\Big\{\beta \Big[\overline{Z}^{(q)}(u)-\overline{Z}^{(q)}(b_{1})\Big]-[u-b_{1}-\alpha]\Big\}}{[Z^{(q)}(u)-Z^{(q)}(b_{1})]^{2}}\label{eq21.0}\\
		&= -\frac{qW^{(q)}(u)}{[Z^{(q)}(u)-Z^{(q)}(b_{1})]^{2}}\bigg\{-\beta\bigg[\frac{[Z^{(q)}(u)]^{2}}{qW^{(q)}(u)}-\overline{Z}^{(q)}(u)-\bigg[\frac{Z^{(q)}(u)Z^{(q)}(b_{1})}{qW^{(q)}(u)}-\overline{Z}^{(q)}(b_{1})\bigg]\bigg]\notag\\
		&\quad+\frac{Z^{(q)}(u)-Z^{(q)}(b_{1})}{qW^{(q)}(u)}-[u-b_{1}-\alpha]\bigg\}\label{eq21.1}\\
		&=-\frac{qW^{(q)}(u)}{Z^{(q)}(u)-Z^{(q)}(b_{1})}[g_{2}(u;b_{1})-\xi(u)]\label{eq21.2}.
	\end{align}
	where $\xi$ is defined in \eqref{zeta}.  Notice that  $u\mapsto[\beta Z^{(q)}(u)-1][Z^{(q)}(u)-Z^{(q)}(b_{1})]$ is strictly increasing and convex on $(b_{1},\infty)$,  and $u\mapsto qW^{(q)}(u)\big\{\beta \big[\overline{Z}^{(q)}(u)-\overline{Z}^{(q)}(b_{1})\big]-[u-b_{1}-\alpha]\big\}$is strictly increasing on $(b_{1},\infty)$. Additionally, they satisfy 
	\begin{align*}
		&\lim_{u\downarrow b_{1}}[\beta Z^{(q)}(u)-1][Z^{(q)}(u)-Z^{(q)}(b_{1})]=0,\notag\\
		&\lim_{u\downarrow b_{1}}qW^{(q)}(u)\big\{\beta \big[\overline{Z}^{(q)}(u)-\overline{Z}^{(q)}(b_{1})\big]-[u-b_{1}-\alpha]\big\}=q\alpha W^{(q)}(b_{1}).
	\end{align*}
	From here and \eqref{eq21.0}, it implies  that there exists $u_{*}>b_{1}$ such that $g'_{2}(u;b_{1})<0$ for $u\in(b_{1},u_{*})$.    On the other hand,  using \eqref{eq21.1}  and following  similar arguments  as in the proof of Proposition 1 in \cite{JMP2019}, it can be shown that there exists $ u^{*}\in[u_{*},\infty)$ such that $g'_{2}(u;b_{1})>0$ for $u\in(u^{*},\infty)$. Then, using the identity \eqref{eq21.2}, we get that  
	\begin{align}\label{i1}
		g_{2}(u;b_{1})>\xi(u)\ \text{for}\ u\in(b_{1},u_{*})\quad\text{and}\quad g_{2}(u;b_{1})<\xi(u)\ \text{for}\ u\in(u^{*},\infty).
	\end{align}	
	
	Define $u^{*}_{b_{1}}= \sup\{u_{*}>b_{1}: g'_{2}(u;b_{1})<0\ \text{for}\ u\in(b_{1},u_{*})\}$ and $\underbar{$u$}\eqdef\inf\{u^{*}>u^{*}_{b_{1}}:g'_{2}(u;b_{1})>0\ \text{for}\ u\in(u^{*},\infty) \}$. If $a^{*}>0$, since $\lim_{u\downarrow b_{1}}g_{2}(u;b_{1})=\infty$ and  $\xi$ is strictly decreasing on $(0,a^{*})$, observe that   $a^{*}\leq u^{*}_{b_{1}}\leq\underbar{$u$}$  must hold due to \eqref{eq21.2} and \eqref{i1}. Otherwise,  i.e. $a^{*}=0$, it obtains easily, by contradiction,  that  $0<u^{*}_{b_{1}}\leq \underbar{$u$}$  we shall show  by contradiction that $u^{*}_{b_{1}}=\underbar{$u$}$, concluding $g_{2}(\cdot,b_{1})$ attains its global minimum at $u^{*}_{b_{1}}$. Assuming that $u^{*}_{b_{1}}<\underbar{$u$}$ and taking $\mathcal{A}$ as the collection of points where $g_{2}(\cdot;b_{1})$ attains a local minimum on $[u^{*}_{b_{1}},\underbar{$u$}]$, notice that $\{u^{*}_{b_{1}},\underbar{$u$}\}\subset\mathcal{A}$, due to   \eqref{eq21.2}, \eqref{i1} and  $g_{2}(\cdot;b_{1}),\xi\in\hol^{2}((0,\infty))$. Additionally, $\mathcal{A}$ cannot contain open sets, because if $(u_{1},u_{2})\subseteq\mathcal{A}$, for some $u_{1},u_{2}\in(u^{*}_{b_{1}},\underbar{$u$})$ such that $u_{1}<u_{2}$, we get that $g'_{2}(u;b_{1})=0$ for $u\in(u_{1},u_{2})$. From here and by \eqref{eq21.2}, it follows that $\xi'(u)=0$ for $u\in(u_{1},u_{2})$,  which is a contradiction since  $\xi'\neq 0$ can only occur  on $(0,\infty)\setminus\{a^{*}\}$. By the previously seen, without a loss of generality, let us consider  $\underbar{$u$}_{1}\in(u^{*}_{b_{1}},\underbar{$u$})$ such that 
	$$g'_{2}(u,b_{1})>0\ \text{for}\   u\in(u^{*}_{b_{1}},\underbar{$u$}_{1} ), \quad g'_{2}(u,b_{1})=0\ \text{at}\ u=\underbar{$u$}_{1}\quad\text{and}\quad g'_{2}(u,b_{1})<0\ \text{for}\   u\in(\underbar{$u$}_{1},\underbar{$u$}), $$ 
	However, using \eqref{eq21.2} and  since $g'_{2}(\underbar{$u$}_{1};b_{1})= g'_{2}(\underbar{$u$};b_{1})=0$ and $\xi$ is strictly increasing on $(a^{*},\infty)$, it yields that $g_{2}(\underbar{$u$}_{1};b_{1})=\xi(\underbar{$u$}_{1})<\xi(\underbar{$u$})=g_{2}(\underbar{$u$})$,
	which is a contradiction, because of $g_{2}(\underbar{$u$}_{1};b_{1})>g_{2}(\underbar{$u$};b_{1})$. 
\end{proof}

\section{Proofs of Lemmas \ref{lem_1} and \ref{lem_2}}\label{D_2}

\begin{proof}[Proof of Lemma \ref{lem_1}]
	Since for each $y\in(0,b^{*}_{2})$ fixed, we get that
	\begin{multline}\label{vp3.3}
		\frac{\der}{\der x}\bigg[\frac{W^{(q+r)}(x-y)}{W^{(q+r)}(x)}\bigg]\\
		=\frac{W^{(q+r)\prime}(x-y)W^{(q+r)\prime}(x)}{[W^{(q+r)}(x)]^{2}}\bigg[\frac{W^{(q+r)}(x)}{W^{(q+r)\prime}(x)}-\frac{W^{(q+r)}(x-y)}{W^{(q+r)\prime}(x-y)}\bigg]>0\quad\text{for}\ x\in(b^{*}_{2},\infty),
	\end{multline}	
	due to the increasing property of $W^{(q+r)}/W^{(q+r)\prime}$ (see Equation (8.18) and Lemma 8.2 in \cite{Kyp}), it is easy to check that 
	$\bar{h}_{1}(\cdot\,;b^{*}_{r},b^{*}_{2})$, $\bar{h}_{2}(\cdot\,;b^{*}_{r},b^{*}_{2})$ are increasing and decreasing on $(b^{*}_{2},\infty)$, respectively. On the other hand, letting $x\rightarrow\infty$, we have that
	\begin{align}
		&\lim_{x\rightarrow\infty}\bar{h}_{1}(x;b^{*}_{r},b^{*}_{2})=r \int_{b_{r}^{*}}^{b^{*}_{2}}\expo^{-\Phi(q+r)y}W^{(q)}(y)\bigg[\frac{g_{1}(b^{*}_{2})}{q\Phi(q+r)}-\xi(y)\bigg]\der y,\label{vp4}\\
		&\lim_{x\rightarrow\infty}\bar{h}_{2}(x;b^{*}_{r},b^{*}_{2})
		=r\bigg[\frac{g_{1}(b^{*}_{2})-g_{1}(b^{*}_{r})}{q\Phi(q+r)}\bigg] \int_{b^{*}_{r}}^{\infty} \expo^{-\Phi(q+r)y}W^{(q)} (y) \der y.\label{vp5}
	\end{align}	
	By proving that the previous limits agree, it follows  that \eqref{vp3.2} is true for any $x$ in $(b^{*}_{2},\infty)$. Using integration by parts, gives
	\begin{align}\label{e5}
		&\expo^{-\Phi(q+r)b^{*}_{2}}Z^{(q)}(b^{*}_{2})-\expo^{-\Phi(q+r)b^{*}_{r}}Z^{(q)}(b^{*}_{r})\notag\\
		&\qquad\qquad=q \int_{b^{*}_{r}}^{b^{*}_{2}}\expo^{-\Phi(q+r)y}W^{(q)}(y)\der y-\Phi(q+r) \int_{b^{*}_{r}}^{b^{*}_{2}}\expo^{-\Phi(q+r)y}Z^{(q)}(y)\der y.
	\end{align}
	Meanwhile, by \eqref{def_z_nuevo} and \eqref{v6}, it is easy to corroborate that for any $b\geq0$ fixed, 
	\begin{align*}
		\beta\expo^{-\Phi(q+r)b} Z^{(q)}(b)
		&=g_{1}(b) \int_{b}^{\infty}\expo^{-\Phi(q+r)y}W^{(q)}(y)\der y+\expo^{-\Phi(q+r)b}-q\beta \int_{b}^{\infty}\expo^{-\Phi(q+r)y}W^{(q)}(y)\der y.
	\end{align*}
	This implies
	\begin{align}\label{e2}
		&\expo^{-\Phi(q+r)b^{*}_{2}}\beta Z^{(q)}(b^{*}_{2})-\expo^{-\Phi(q+r)b^{*}_{r}}\beta Z^{(q)}(b^{*}_{r})=\notag\\
		&[g_{1}(b^{*}_{2})-g_{1}(b^{*}_{r})] \int_{b^{*}_{r}}^{\infty}\expo^{-\Phi(q+r)y}W^{(q)}(y)\der y-g_{1}(b^{*}_{2}) \int_{b^{*}_{r}}^{b^{*}_{2}}\expo^{-\Phi(q+r)y}W^{(q)}(y)\der y\notag\\
		&-\Phi(q+r) \int_{b^{*}_{r}}^{b^{*}_{2}}\expo^{-\Phi(q+r)y}\der y+q\beta \int_{b^{*}_{r}}^{b^{*}_{2}}\expo^{-\Phi(q+r)y}W^{(q)}(y)\der y.
	\end{align}
	Using \eqref{e5} and \eqref{e2}, it is easy to check that \eqref{vp4} and \eqref{vp5} are equal. 
\end{proof}

\begin{proof}[Proof of Lemma \ref{lem_2}]
	 Notice that $a_{1}$ and $a_{2}$ are positive,  because of \eqref{e2.2}. Observe that $\bar{h}_{3}$ is decreasing  on $(b^{*}_{2},\infty)$, due to \eqref{vp3.3}. Additionally,  $\bar{h}_{3}$ fulfils
	\begin{equation}\label{lim1}
		\lim_{x\rightarrow\infty}\bar{h}_{3}(x)=a_{1}\bigg[1-r\int_{0}^{b^{*}_{2}}\expo^{-\Phi(q+r)y}W^{(q)} (y) \der y\bigg].
	\end{equation}
	Meanwhile, since $\frac{W^{(q+r)}(x)}{W^{(q+r)\prime}(x)}\geq\frac{W^{(q+r)}(y)}{W^{(q+r)\prime}(y)}$  for $y<x-b^{*}_{2}<x$, we get that  $\frac{W^{(q+r)}(x)}{W^{(q+r)\prime}(x)}\geq\frac{\overline{W}^{(q+r)}(x-b^{*}_{2})}{W^{(q+r)}(x-b^{*}_{2})}$. From here, it follows that 
	\begin{equation*}
		\frac{\der}{\der x}\bar{h}_{4}(x)=ra_{2}\frac{W^{(q+r)}(x-b^{*}_{2})W^{(q+r)\prime}(x)}{[W^{(q+r)}(x)]^{2}}\bigg[\frac{W^{(q+r)}(x)}{W^{(q+r)\prime}(x)}-\frac{\overline{W}^{(q+r)}(x-b^{*}_{2})}{W^{(q+r)}(x-b^{*}_{2})}\bigg]\geq0\quad\text{for}\ x>b^{*}_{2}.
	\end{equation*}
	Thus, $\bar{h}_{4}$ is increasing on $(b^{*}_{2},\infty)$. Moreover,
	\begin{align}\label{lim2}
		\lim_{x\rightarrow\infty}\bar{h}_{4}(x)=a_{2}r\int_{b^{*}_{2}}^{\infty}\expo^{-\Phi(q+r)y}\der y=\frac{ra_{2}}{\Phi(q+r)}\expo^{-\Phi(q+r)b^{*}_{2}}.
	\end{align}
	Using \eqref{def_z_nuevo}, \eqref{eq5.0}, \eqref{eq16.6}, \eqref{e2.6}, \eqref{lim1}, and \eqref{lim2}, we get that \eqref{vp3.2.0} is true. 
\end{proof}
\section{Proofs of Proposition \ref{L.A.} and Lemmas \ref{L.A.2}--\ref{L.A.3} }\label{D}

\begin{proof}[Proof of Proposition \ref{L.A.}]
	If $x<b_{2}$, by strong Markov property and using \eqref{laplace_in_terms_of_z}, we get
	\begin{align}\label{eq8}
		g(x;\br,c,\theta)&=\E_{x}\Big[\expo^{-q\tau^{+}_{b_{2}}};\tau^{+}_{b_{2}}<\tau^{-}_{0}\Big]g(b_{2};\br,c,\theta)+\E_{x}\Big[\expo^{-(q\tau^{-}_{0}-\theta X(\tau^{-}_{0}))};\tau^{-}_{0}<\tau^{+}_{b_{2}}\Big]\notag\\
		&=\frac{W^{(q)}(x)}{W^{(q)}(b_{2})}g(b_{2};\br,c,\theta)+Z^{(q)}(x,\theta)-Z^{(q)}(b_{2},\theta)\frac{W^{(q)}(x)}{W^{(q)}(b_{2})}\notag\\
		&=Z^{(q)}(x,\theta)+\frac{W^{(q)}(x)}{W^{(q)}(b_{2})}[g(b_{2};\br,c,\theta)-Z^{(q)}(b_{2},\theta)].
	\end{align}
	If $b_{2}<x\leq c$, then, using again strong Markov property,
	\begin{align}\label{g.3}
		g(x;\br,c,\theta)&=g_{1}(x;\br,c,\theta)+g_{2}(x;\br,c,\theta),
	\end{align}
	with
	\begin{align*}
		g_{1}(x;\br,c,\theta)&\eqdef\E_{x}\Big[\expo^{-q\tau^{-}_{b_{2}}}g(X(\tau^{-}_{b_{2}});\br,c,\theta);\tau^{-}_{b_{2}}<\tau^{+}_{c}\wedge T(1)\Big],\\
		g_{2}(x;\br,c,\theta)&\eqdef \E_{x}\Big[\expo^{-qT(1)};T(1)<\tau^{-}_{b_{2}}\wedge\tau^{+}_{c}\Big]g(b_{1};\br,0,c,\theta).
	\end{align*}
	
	Using \eqref{id.0} and \eqref{eq8}, it gives
	\begin{align}
		g_{1}(x;\br,c,\theta)
		&=\E_{x}\Big[\expo^{-(q+r)\tau^{-}_{b_{2}}}g(X(\tau^{-}_{b_{2}});\br,c,\theta);\tau^{-}_{b_{2}}<\tau^{+}_{c}\Big]\notag\\
		&=\E_{x}\Big[\expo^{-(q+r)\tau^{-}_{b_{2}}}Z^{(q)}(X(\tau^{-}_{b_{2}}),\theta);\tau^{-}_{b_{2}}<\tau^{+}_{c}\Big]\notag\\
		&\quad+\frac{[g(b_{2};\br,c,\theta)-Z^{(q)}(b_{2},\theta)]}{W^{(q)}(b_{2})}\E_{x}\Big[\expo^{-(q+r)\tau^{-}_{b_{2}}}W^{(q)}(X(\tau^{-}_{b_{2}}));\tau^{-}_{b_{2}}<\tau^{+}_{c}\Big]\notag\\
		&=Z^{(q,r)}_{b_{2}}(x,\theta)+\frac{[g(b_{2};\br,c,\theta)-Z^{(q)}(b_{2},\theta)]}{W^{(q)}(b_{2})}W^{(q,r)}_{b_{2}}(x)\notag\\
		&\quad-\frac{W^{(q+r)}(x-b_{2})}{W^{(q+r)}(c-b_{2})}\bigg[Z^{(q,r)}_{b_{2}}(c,\theta)+\frac{[g(b_{2};\br,c,\theta)-Z^{(q)}(b_{2},\theta)]}{W^{(q)}(b_{2})}W^{(q,r)}_{b_{2}}(c)\bigg].\label{g.1}
	\end{align}	
Using \eqref{N.1} and \eqref{eq8}, we get 
	\begin{align}
		&g_{2}(x;\br,c,\theta)\notag\\
		&=r\frac{\overline{W}^{(q+r)}(c-b_{2})}{W^{(q+r)}(c-b_{2})}W^{(q+r)}(x-b_{2})\bigg[Z^{(q)}(b_{1},\theta)+\frac{W^{(q)}(b_{1})}{W^{(q)}(b_{2})}[g(b_{2};\br,c,\theta)-Z^{(q)}(b_{2},\theta)]\bigg]\notag\\
		&\quad-r\overline{W}^{(q+r)}(x-b_{2})\bigg[Z^{(q)}(b_{1},\theta)+\frac{W^{(q)}(b_{1})}{W^{(q)}(b_{2})}[g(b_{2};\br,c,\theta)-Z^{(q)}(b_{2},\theta)]\bigg].\label{g.2}
	\end{align}
Applying \eqref{g.1}--\eqref{g.2} in  \eqref{g.3}, it gives the following identity,
	\begin{align}\label{eq9}
		g(x;\br,c,\theta)
		&=Z^{(q,r)}_{b_{2}}(x,\theta)+\frac{[g(b_{2};\br,c,\theta)-Z^{(q)}(b_{2},\theta)]}{W^{(q)}(b_{2})}W^{(q,r)}_{b_{2}}(x)\notag\\
		&\quad-r\overline{W}^{(q+r)}(x-b_{2})\bigg[Z^{(q)}(b_{1},\theta)+\frac{W^{(q)}(b_{1})}{W^{(q)}(b_{2})}[g(b_{2};\br,c,\theta)-Z^{(q)}(b_{2},\theta)]\bigg]\notag\\
		&\quad+\frac{W^{(q+r)}(x-b_{2})}{W^{(q+r)}(c-b_{2})}\bigg\{r\overline{W}^{(q+r)}(c-b_{2})Z^{(q)}(b_{1},\theta)-Z^{(q,r)}_{b}(c,\theta)\notag\\
		&\quad+\frac{[g(b_{2};\br,c,\theta)-Z^{(q)}(b_{2},\theta)]}{W^{(q)}(b_{2})}[r\overline{W}^{(q+r)}(c-b_{2})W^{(q)}(b_{1})-W^{(q,r)}_{b_{2}}(c)]\bigg\}.
	\end{align}
	Let us verify that the fourth term that is on the right side of \eqref{eq9} is equal to zero. Observe that
	\begin{align}\label{eq9.1}
		g(b_{2};\br,c,\theta)&=\E_{b_{2}}\Big[\expo^{-(q\tau^{-}_{0}(r;\br)-\theta U^{\br}_{r}(\tau^{-}_{0}(r;\br)))}\uno_{\{\tau^{-}_{0}(r;\br)<\tau^{-}_{c}(r;\br)\}};T(1)<\tau^{-}_{0}\wedge\tau^{+}_{c}\Big]\notag\\
		&\quad+\E_{b_{2}}\Big[\expo^{-(q\tau^{-}_{0}(r;\br)-\theta U^{\br}_{r}(\tau^{-}_{0}(r;\br)))}\uno_{\{\tau^{-}_{0}(r;\br)<\tau^{-}_{c}(r;\br)\}};\tau^{-}_{0}<T(1)\wedge\tau^{+}_{c}\Big]\defeq\delta_{1}+\delta_{2}.
	\end{align}
	By the strong Markov property, we get that
	\begin{align}
		\delta_{1}&=	\E_{b_{2}}\Big[\expo^{-(q\tau^{-}_{0}(r;\br)-\theta U^{\br}_{r}(\tau^{-}_{0}(r;\br)))}\uno_{\{\tau^{-}_{0}(r;\br)<\tau^{-}_{c}(r;\br)\}};T(1)<\tau^{-}_{0}\wedge\tau^{+}_{c},b_{2}\leq X(T(1))\leq c\Big]\notag\\
		&\quad+\E_{b_{2}}\Big[\expo^{-(q\tau^{-}_{0}(r;\br)-\theta U^{\br}_{r}(\tau^{-}_{0}(r;\br)))}\uno_{\{\tau^{-}_{0}(r;\br)<\tau^{-}_{c}(r;\br)\}};T(1)<\tau^{-}_{0}\wedge\tau^{+}_{c},   0< X(T(1))< b_{2}\Big]\notag\\
		&=\E_{b_{2}}\Big[\expo^{-qT(1)};T(1)<\tau^{-}_{0}\wedge\tau^{+}_{c},b_{2}\leq X(T(1))\leq c\Big]g(b_{1};\br,c,\theta)\notag\\
		&\quad+\E_{b_{2}}\Big[\expo^{-qT(1)}g(X(T(1));\br,c,\theta);T(1)<\tau^{-}_{0}\wedge\tau^{+}_{c},   0< X(T(1))< b_{2}\Big]\notag\\
		&=r\E_{b_{2}}\bigg[ \int_{0}^{\tau^{-}_{0}\wedge\tau^{+}_{c}}\expo^{-(q+r)\tmt}\uno_{\{b_{2}\leq X(\tmt)\leq c\}}\der \tmt\bigg]g(b_{1};\br,c,\theta)\notag\\
		&\quad+r\E_{b_{2}}\bigg[ \int_{0}^{\tau^{-}_{0}\wedge\tau^{+}_{c}}\expo^{-(q+r)\tmt}g(X(\tmt);\br,c,\theta)\uno_{ \{0< X(\tmt)< b_{2}\}}\der \tmt\bigg].\label{eq11.1}
	\end{align}
	Observe that
	\begin{align}\label{eq11}
		r\E_{b_{2}}\bigg[ \int_{0}^{\tau^{-}_{0}\wedge\tau^{+}_{c}}\expo^{-(q+r) \tmt}\uno_{\{b_{2}\leq X(\tmt)\leq c\}}\der \tmt\bigg]
		&=r\frac{W^{(q+r)}(b_{2})}{W^{(q+r)}(c)}\overline{W}^{(q+r)}(c-b_{2}).
	\end{align}	
due to \eqref{res.1}. Meanwhile,  since
	\begin{align*}
  &r \int_{0}^{b_{2}}W^{(q)}(y)W^{(q+r)}(c-y)\der y=W^{(q+r)}(c)-W_{b_{2}}^{(q,r)}(c),\notag\\
		&r \int_{0}^{b_{2}}Z^{(q)}(y,\theta)W^{(q+r)}(c-y)\der y=Z^{(q+r)}(c,\theta)-Z_{b_{2}}^{(q,r)}(c,\theta),
	\end{align*}
	because of \eqref{eq1}--\eqref{identities}, and by \eqref{res.1} and \eqref{eq8}, it implies
	\begin{align}
		&r\E_{b_{2}}\bigg[ \int_{0}^{\tau^{-}_{0}\wedge\tau^{+}_{c}}\expo^{-(q+r)\tmt}g(X(\tmt);\br,c,\theta)\uno_{ \{0< X(\tmt)< b_{2}\}}\der \tmt\bigg]\notag\\
		&=\frac{[g(b_{2};\br,c,\theta)-Z^{(q)}(b_{2},\theta)]}{W^{(q)}(b_{2})}\notag\\
		&\quad\times\bigg[\frac{W^{(q+r)}(b_{2})}{W^{(q+r)}(c)}r \int_{0}^{b_{2}}W^{(q)}(y)W^{(q+r)}(c-y)\der y-r \int_{0}^{b_{2}}W^{(q)}(y)W^{(q+r)}(b_{2}-y)\der y\bigg]\notag\\
		&\quad+\frac{W^{(q+r)}(b_{2})}{W^{(q+r)}(c)}r \int_{0}^{b_{2}}Z^{(q)}(y,\theta)W^{(q+r)}(c-y)\der y-r \int_{0}^{b_{2}}Z^{(q)}(y,\theta)W^{(q+r)}(b_{2}-y)\der y.\notag\\
		&=\frac{[g(b_{2};\br,c,\theta)-Z^{(q)}(b_{2},\theta)]}{W^{(q)}(b_{2})}\bigg[\frac{W^{(q+r)}(b_{2})}{W^{(q+r)}(c)}[W^{(q+r)}(c)-W_{b_{2}}^{(q,r)}(c)]-[W^{(q+r)}(b_{2})-W^{(q)}(b_{2})]\bigg]\notag\\
		&\qquad+\frac{W^{(q+r)}(b_{2})}{W^{(q+r)}(c)}[Z^{(q+r)}(c,\theta)-Z_{b_{2}}^{(q,r)}(c,\theta)]-[Z^{(q+r)}(b_{2},\theta)-Z^{(q)}(b_{2},\theta)]\notag\\
		&\quad=\frac{W^{(q+r)}(b_{2})}{W^{(q+r)}(c)}\bigg\{-Z^{(q,r)}_{b_{2}}(c,\theta)-\frac{[g(b_{2};\br,c,\theta)-Z^{(q)}(b_{2},\theta)]}{W^{(q)}(b_{2})}W_{b_{2}}^{(q,r)}(c)\bigg\}\notag\\
		&\qquad+g(b_{2};\br,c,\theta)+\frac{W^{(q+r)}(b_{2})}{W^{(q+r)}(c)}Z^{(q+r)}(c,\theta)-Z^{(q+r)}(b_{2},\theta).\label{eq11.0}
	\end{align}
	Thus, putting \eqref{eq11}--\eqref{eq11.0} in \eqref{eq11.1}, it yields that
	\begin{align}
		\delta_{1}
		&=\frac{W^{(q+r)}(b_{2})}{W^{(q+r)}(c)}\bigg\{r\overline{W}^{(q+r)}(c-b_{2})Z^{(q)}(b_{1},\theta)-Z^{(q,r)}_{b_{2}}(c,\theta)\notag\\
		&\quad+\frac{[g(b_{2};\br,c,\theta)-Z^{(q)}(b_{2},\theta)]}{{W^{(q)}(b_{2})}}[r\overline{W}^{(q+r)}(c-b_{2})W^{(q)}(b_{1})-W_{b_{2}}^{(q,r)}(c)]\bigg\}\notag\\
		&\qquad+g(b_{2};\br,c,\theta)+\frac{W^{(q+r)}(b_{2})}{W^{(q+r)}(c)}Z^{(q+r)}(c,\theta)-Z^{(q+r)}(b_{2},\theta).\label{eq12}
	\end{align}
	Meanwhile, by \eqref{laplace_in_terms_of_z},
	\begin{align}
		\delta_{2}&=\E_{b_{2}}\Big[\expo^{-(q+r)\tau^{-}_{0}+\theta X(\tau^{-}_{0})};\tau^{-}_{0}<\tau^{+}_{c}\Big]=Z^{(q+r)}(b_{2},\theta)-\frac{W^{(q+r)}(b_{2})}{W^{(q+r)}(c)}Z^{(q+r)}(c,\theta).\label{eq13}
	\end{align}
	Applying  \eqref{eq12}--\eqref{eq13} in \eqref{eq9.1}, it follows that 
	\begin{align*}
		0&=	r\overline{W}^{(q+r)}(c-b_{2})Z^{(q)}(b_{1},\theta)-Z^{(q,r)}_{b_{2}}(c,\theta)\notag\\
		&\quad+\frac{[g(b_{2};\br,c,\theta)-Z^{(q)}(b_{2},\theta)]}{{W^{(q)}(b_{2})}}[r\overline{W}^{(q+r)}(c-b_{2})W^{(q)}(b_{1})-W_{b_{2}}^{(q,r)}(c)].
	\end{align*}	
	Therefore, from here  and \eqref{eq9}, it is easy to obtain the expressions given  in  \eqref{gb.1}--\eqref{gb.2}.
\end{proof}

\begin{proof}[Proof of Lemma \ref{L.A.2}]
Letting $c\uparrow\infty$ in \eqref{gb.1}, it gives that 
\begin{multline}
	E_{x}\Big[\expo^{-(q\tau^{-}_{0}(r;\br)-\theta U^{\br}_{r}(\tau^{-}_{0}(r;\br)))}; \tau^{-}_{0}(r;\br)<\infty\Big]
	=Z^{(q,r)}_{b_{2}}(x,\theta)-r\overline{W}^{(q+r)}(x-b_{2})Z^{(q)}(b_{1},\theta)\\
	-[W^{(q,r)}_{b_{2}}(x)-r\overline{W}^{(q+r)}(x-b_{2})W^{(q)}(b_{1})]\lim_{c\uparrow\infty}C^{(q,r)}_{\br,c}(\theta).\label{gb.1.0}
\end{multline}
So, we  only need to analyse the limit of the expression
\begin{align}\label{C.1}
	C^{(q,r)}_{\br,c}(\theta)
	&=\frac{\frac{Z^{(q,r)}_{b_{2}}(c,\theta)}{W^{(q+r)}(c-b_{2})}-r\frac{\overline{W}^{(q+r)}(c-b_{2})}{W^{(q+r)}(c-b_{2})}Z^{(q)}(b_{1},\theta)}{\frac{W_{b_{2}}^{(q,r)}(c)}{W^{(q+r)}(c-b_{2})}-r\frac{\overline{W}^{(q+r)}(c-b_{2})}{W^{(q+r)}(c-b_{2})}W^{(q)}(b_{1})},
\end{align}
when $c\rightarrow\infty$. Taking into account \eqref{W1}--\eqref{W2} and using L'Hopital's rule, it is easy to verify that
 \begin{align}\label{C.2}
	\lim_{c\uparrow\infty}\frac{\overline{W}^{(q+r)}(c-b_{2})}{W^{(q+r)}(c-b_{2})}&=\frac{1}{\Phi(q+r)}.
\end{align}	
Meanwhile, using \eqref{scale_function_laplace}, \eqref{identities} and Fubini's theorem, it can be verified
\begin{align}\label{C.3}
	r \int_{b_{2}}^{\infty}\expo^{-\Phi(q+r)(y-b_{2})}Z^{(q)}(y,\theta)\der y
	&=\frac{rZ^{(q)}(b_{2},\theta)}{\Phi(q+r)-\theta}+\bigg[\frac{q-\psi(\theta)}{\Phi(q+r)-\theta}\bigg]Z^{(q)}(b_{2},\Phi(q+r)).
\end{align}
Letting $c\uparrow\infty$ in \eqref{C.1} and taking into account \eqref{id.0.2} and \eqref{C.2}--\eqref{C.3}, we see that  $\overline{C}^{(q,r)}_{\br}(\theta)=\lim_{c\uparrow\infty}C^{(q,r)}_{\br,c}(\theta)$ is as in \eqref{eq15.0}. From here and \eqref{gb.1.0}, we obtain the expression in \eqref{eq15}.
\end{proof}
\begin{proof}[Proof of Lemma \ref{L.A.3}]
Since $\lim_{\theta\downarrow0}Z^{(q,r)}(x,\theta)=Z^{(q,r)}(x)$ and $\lim_{\theta\downarrow0}Z^{(q)}(x,\theta)=Z^{(q)}(x)$ for $x\in\R$, due to \eqref{z_t} and \eqref{identities}, and  letting $\theta\downarrow0$ in \eqref{eq15}, it is easy to corroborate that \eqref{eq16} is true,  where $\overline{C}^{(q,r)}_{\br}=\lim_{\theta\downarrow0}\overline{C}^{(q,r)}_{\br}(\theta)$  and $\overline{C}^{(q,r)}_{\br}(\theta)$ is as  in \eqref{eq15.0}.
Taking first derivative w.r.t. $\theta$ in \eqref{eq15}, observe that 
\begin{align}\label{eq17}
	&\E_{x}\Big[\expo^{-(q\tau^{-}_{0}(r;\br)-\theta U^{\br}_{r}(\tau^{-}_{0}(r;\br)))} U^{\br}_{r}(\tau^{-}_{0}(r;\br)); \tau^{-}_{0}(r;\br)<\infty\Big]\notag\\
	&\quad\qquad=\partial_{\theta}[Z^{(q,r)}_{b_{2}}(x,\theta)]-r\overline{W}^{(q+r)}(x-b_{2})\partial_{\theta}[Z^{(q)}(b_{1},\theta)]\notag\\
	&\qquad\qquad-\partial_{\theta}[\overline{C}^{(q,r)}_{\br}(\theta)][W^{(q,r)}_{b_{2}}(x)-r\overline{W}^{(q+r)}(x-b_{2})W^{(q)}(b_{1})],
\end{align}
where for $\theta\in\R$,
\begin{align}
	\partial_{\theta}[Z^{(q)}(x,\theta)]
	&=xZ^{(q)}(x,\theta)-\expo^{\theta x}[q-\psi(\theta)] \int_{0}^{x}z\expo^{-\theta z}W^{(q)}(z)\der z\notag\\
	&\quad-\expo^{\theta x}\psi'(\theta) \int_{0}^{x}\expo^{-\theta z}W^{(q)}(z)\der z,\label{eq17.0}\\
	\partial_{\theta}[Z^{(q,r)}_{b_{2}}(x,\theta)]&=\partial_{\theta}[Z^{(q)}(x,\theta)]+r \int_{b_{2}}^{x}W^{(q+r)}(x-y)\partial_{\theta}[Z^{(q)}(y,\theta)]\der y\label{eq17.0.0}
\end{align}
and
\begin{align}
	\partial_{\theta}[\overline{C}^{(q,r)}_{\br}(\theta)]
	&=\frac{1}{Z^{(q)}(b_{2},\Phi(q+r))-\frac{r}{\Phi(q+r)}W^{(q)}(b_{1})}\notag\\
	&\hspace{-2cm}\times\bigg\{\frac{r}{[\Phi(q+r)-\theta]^2}Z^{(q)}(b_{2},\theta)+\frac{r}{\Phi(q+r)-\theta}\partial_{\theta}[Z^{(q)}(b_{2},\theta)]\notag\\
	&\hspace{-2cm}+\bigg[\frac{-\psi'(\theta)[\Phi(q+r)-\theta]+q-\psi(\theta)}{[\Phi(q+r)-\theta]^{2}}\bigg]Z^{(q)}(b_{2},\Phi(q+r))-\frac{r}{\Phi(q+r)}\partial_{\theta}[Z^{(q)}(b_{1},\theta)]\bigg\}.\label{eq17.1}
\end{align}
By \eqref{v2} and using integration by parts, notice that 
\begin{align*}
	\partial_{\theta}[Z^{(q)}(x,\theta)]\Big|_{\theta=0}&=xZ^{(q)}(x)-q \int_{0}^{x}zW^{(q)}(z)\der z-\psi'(0+) \int_{0}^{x}W^{(q)}(z)\der z=l^{(q)}(x).
\end{align*}
The previous relation implies that  $\partial_{\theta}[Z^{(q,r)}_{b_{2}}(x,\theta)]\Big|_{\theta=0}=l^{(q,r)}_{b_{2}}(x)$, with $l^{(q,r)}_{b_{2}}$ as in \eqref{eq16.1}. Applying \eqref{eq17.0}--\eqref{eq17.0.0} in \eqref{eq17.1}, it follows that $\widehat{C}_{\br}^{(q,r)}=\partial_{\theta}[{C}_{\br}^{(q,r)}(\theta)]\Big|_{\theta=0}$ is as in \eqref{eq16.0}. Putting the pieces together and letting $\theta\downarrow0$ in \eqref{eq17}, we obtain that 
\begin{align*}
	\E_{x}&\Big[\expo^{-q\tau^{-}_{0}(r;\br)} U^{\br}_{r}(\tau^{-}_{0}(r;\br)); \tau^{-}_{0}(r;\br)<\infty\Big]\notag\\
	&=K^{(q,r)}_{b_{2}}(x)+r\overline{W}^{(q+r)}(x-b_{2})[l^{(q)}(b_{2})-l^{(q)}(b_{1})]\notag\\
	&\quad-\widehat{C}^{(q,r)}_{\br}[W^{(q,r)}_{b_{2}}(x)-r\overline{W}^{(q+r)}(x-b_{2})W^{(q)}(b_{1})].
\end{align*}
Therefore, from here and by \eqref{eq16.4},  it is easy to obtain \eqref{eq16.2}. 
\end{proof}

\end{document}